\documentclass[onefignum,onetabnum]{siamart190516}

\def\titlename{Graph topology invariant gradient and sampling complexity for decentralized and stochastic optimization}

\usepackage{lipsum}
\usepackage{amsfonts}
\usepackage{graphicx}
\usepackage{epstopdf}

\ifpdf
\DeclareGraphicsExtensions{.eps,.pdf,.png,.jpg}
\else
\DeclareGraphicsExtensions{.eps}
\fi

\newsiamremark{remark}{Remark}
\newsiamremark{hypothesis}{Hypothesis}
\crefname{hypothesis}{Hypothesis}{Hypotheses}
\newsiamthm{claim}{Claim}

\headers{Topology invariant gradient and sampling complexity}{Guanghui Lan, Yuyuan Ouyang, and Yi Zhou}

\title{\titlename\thanks{Submitted to the editors DATE.
		\funding{The first author is partially supported by the Office of Navel Research grant N00014-20-1-2089 and Army Research Office grant W911NF-18-1-0223. The second author is partially supported by the Office of Navel Research grants N00014-20-1-2089 and N00014-19-1-2295.}}}

\author{Guanghui Lan\thanks{H. Milton Stewart School of Industrial and Systems Engineering, Georgia Institute of Technology, Atlanta, GA
		(\email{george.lan@isye.gatech.edu}
		).}
	\and Yuyuan Ouyang\thanks{School of Mathematical and Statistical Sciences, Clemson University, Clemson, SC
		(\email{yuyuano@clemson.edu}
		).}
	\and Yi Zhou\thanks{IBM Almaden Research Center, San Jose, CA
		(\email{yi.zhou@ibm.com}
		).}}

\usepackage{amsopn}

\ifpdf
\hypersetup{
  pdftitle={\titlename},
  pdfauthor={Guanghui Lan, Yuyuan Ouyang, and Yi Zhou}
}
\fi

\usepackage{yuyuan}
\usepackage{autonum}
\usepackage{algorithm, algpseudocode, multirow}
\usepackage{braket}
\usepackage{enumitem}

\def\kp{\tp[k]}

\def\tf{\tilde{f}}
\def\tL{\tilde{L}}
\def\tu{\tilde{u}}

\def\tx{\tilde{x}}

\def\hx{\hat{x}}

\def\hz{\hat{z}}
\def\hw{\hat{w}}

\def\xl{\underline{x}}
\def\xu{\overline{x}}

\def\yu{\overline{y}}

\def\zu{\overline{z}}

\def\wu{\overline{w}}
\def\tp{^{t-1}}

\def\kp{_{k-1}}
\def\kn{_{k+1}}

\newcommand{\tsum}{\textstyle{\sum}}

\begin{document}

\maketitle

\begin{abstract}
	One fundamental problem in constrained decentralized multi-agent optimization is the trade-off between gradient/sampling complexity and communication complexity.
	In this paper we propose new algorithms whose gradient and sampling complexities are graph topology invariant, while their communication complexities remain optimal. Specifically, for convex smooth deterministic problems, we propose a primal dual sliding (PDS) algorithm that is able to compute an $\varepsilon$-solution with $\cO((\tL/\varepsilon)^{1/2})$ gradient complexity and $\cO((\tL/\varepsilon)^{1/2}+\|\cA\|/\varepsilon)$ communication complexity, where $\tL$ is the smoothness parameter of the objective function and $\cA$ is related to either the graph Laplacian or the transpose of the oriented incidence matrix of the communication network. The complexities can be further improved to $\cO((\tL/\mu)^{1/2}\log(1/\varepsilon))$ and $\cO((\tL/\mu)^{1/2}\log(1/\varepsilon) + \|\cA\|/\varepsilon^{1/2})$ respectively with the additional assumption of strong convexity modulus $\mu$. We also propose a stochastic variant, namely the primal dual sliding (SPDS) algorithm for convex smooth problems with stochastic gradients. The SPDS algorithm utilizes the mini-batch technique and enables the agents to perform sampling and communication simultaneously. It computes a stochastic $\varepsilon$-solution with $\cO((\tL/\varepsilon)^{1/2} + (\sigma/\varepsilon)^2)$ sampling complexity, which can be further improved to $\cO((\tL/\mu)^{1/2}\log(1/\varepsilon) + \sigma^2/\varepsilon)$ in the strong convexity case. Here  $\sigma^2$ is the variance of the stochastic gradient. The communication complexities of SPDS remain the same as that of the deterministic case. All the aforementioned gradient and sampling complexities match the lower complexity bounds for centralized convex smooth optimization and are independent of the network structure. 
	To the best of our knowledge, these gradient and sampling complexities have not been obtained before in the literature of decentralized optimization over a constraint feasible set.
\end{abstract}



\section{Introduction}

The problem of interest in this paper is the following decentralized multi-agent optimization problem in which the agents will collaboratively minimize an overall objective as the sum of all local objective $f_i$'s.:
\begin{equation}
	\label{eq:problem_multi_agent}
	\min_{x\in X} \textstyle\sum_{i=1}^{m}f_i(x).
\end{equation}
Here each $f_i:X^{(i)}\to\R$ is a convex smooth function defined over a closed convex set $X^{(i)}\in\R^{d}$ and $X:=\cap_{i=1}^mX^{(i)}$. Under decentralized settings, each agent is expected to collect data, perform numerical operations using local data, and pass information to the neighboring agents in a communication network; no agent has full knowledge about other agents' local objectives or the communication network. 
This type of decentralized problems has many applications in signal processing, control, robust statistical inference, machine learning among others (see, e.g., \cite{durham2012distributed,jadbabaie2003coordination,rabbat2004distributed, ram2009distributed, nedic2017fast, forero2010consensus}).

In this paper, we assume that the communication network between the agents are defined by a connected undirected graph $\cG = (\cN,\cE)$, where $\cN=\{1,\ldots, m\}$ is the set of indices of agents and $\cE\subseteq \cN\times \cN$ is the set of communication edges between them. Each agent $i\in\cN$ could directly communicate with the agents in its neighborhood 
$	N_i:=\Set{j\in\cN|(i,j)\in\cE}\cup\{i\}$.
For convenience we assume that there exists a loop $(i,i)$ for all agents $i\in\cN$.
In addition, we assume that the local objectives are all large-dimensional functions and the $i$-th agent may only be able to access the information of its local objective $f_i(x)$ through a sampling process of its first-order information, 
e.g., the agent has objective $f_i(x)=\E_{\xi_i}[F_i(x, \xi_i)]$, where the expectation is taken with respect to the random variable $\xi_i$ and the distribution of $\xi_i$ is not known in advance. In order to solve problem \eqref{eq:problem_multi_agent} collaboratively, the agents have to sample the first-order information of their own local objective $f_i$ and also communicate with the neighbors in the communication network in order to reach consensus. Our goal in this paper is to design an algorithm that is efficient, in terms of both the sampling and communication complexity, on solving the multi-agent decentralized problem \eqref{eq:problem_multi_agent}.

\subsection{Problem formulation and assumptions}
The aforementioned multi-agent problem can also be formulated as the following linearly constrained problem:
\begin{equation}
	\label{eq:problem_decen}
	\min_{x\in \cX} f(x):=\tsum_{i=1}^{m}f_i(x^{(i)}) \st \cA x = 0,
\end{equation}
where the overall objective function $f:\cX\to\R$ is defined on the Cartesian product $\cX:=X^{(1)}\times \cdots\times X^{(m)}\subseteq\R^{md}$ and we use the notation $x = (x^{(1)},\ldots, x^{(m)})\in\cX$ to summarize a collection of local solutions $x^{(i)}\in X^{(i)}$. The matrix $\cA$ in the linear constraint $\cA x=0$ is designed to enforce
the conditions $x^{(i)}=x^{(j)}$ for all agents $i$ and $j$ that are connected by a communication edge. One possible choice of $\cA$ is $\cA = \cL\otimes I_d\in\R^{md\times md}$ where $I_d\in\R^{d\times d}$ is an identity matrix and $\cL\in\R^{m\times m}$ is the graph Laplacian matrix whose $(i,j)$-entry is $|N_i|-1$ if $i=j$, $-1$ if $i\not=j \text{ and }(i,j)\in\cE$, and $0$ otherwise. 
Here $|N_i|$ is 
the degree of node $i$.
One other choice of $\cA$ is $\cA = \cB^\top\otimes I_d\in\R^{|\cE|d\times md}$ where $\cB\in\R^{m\times |\cE|}$ is an oriented incidence matrix of graph $\cG$. Here $|\cE|$ is the number of edges in $\cG$.
For both cases, it can be shown that $\cA x=0$ if and only if $x^{(i)}=x^{(j)}$ for all agents $i$ and $j$ that are connected via a communication edge. 
 Note that problem \eqref{eq:problem_decen} is equivalent to the saddle point problem
\begin{equation}
	\label{eq:spp_decen}
	\min_{x\in \cX}\max_{z\in\R^{md}} f(x) + \langle \cA x, z\rangle.
\end{equation}
We assume that there exists a saddle point $(x^*,z^*)$ for problem \eqref{eq:spp_decen}.

In this paper, our goal is to solve the multi-agent problem \eqref{eq:problem_multi_agent} by obtaining an $\varepsilon$-approximate solution $\xu$ to the linearly constrained problem \eqref{eq:problem_decen} such that $f(\xu) - f(x^*)\le\varepsilon$ and $\|\cA\xu\|_2\le\varepsilon$.
We make the following assumptions on the information and sampling process of the local objective functions. For each $i$, we assume that
\begin{equation}
	\label{eq:tf}
	f_i(x^{(i)}) := \tf_i(x^{(i)}) + \mu\nu_i(x^{(i)})
\end{equation}
where $\mu\ge 0$, $\tf_i$ is a convex smooth function, and $\nu_i$ is a strongly convex function with strong convexity modulus $1$ with respect to a norm $\|\cdot\|$. Here the gradient $\nabla\tf_i$ is Lipschitz continuous with constant $\tL$ with respect to the norm $\|\cdot\|$.
%
%
The $i$-th agent can access first-order information of $\tf_i$ through a stochastic oracle which returns an unbiased gradient estimator $G_i(x^{(i)}, \xi^{(i)})$ for any inquiry point $x^{(i)}\in X^{(i)}$, where $\xi^{(i)}$ is a sample of an underlying random vector with unknown distribution. The unbiased gradient estimator satisfies $\E[G_i(x^{(i)},\xi^{(i)})] = \nabla f_i(x^{(i)})$ and
\begin{equation}
	\label{eq:assum_so}
	\E[\|G_i(x^{(i)},\xi^{(i)}) - \nabla f_i(x^{(i)})\|_2^2]\le \sigma^2,\ \forall x^{(i)}\in X^{(i)},
\end{equation}
where $\|\cdot\|_2$ is the Euclidean norm. Note that the above assumption also covers the deterministic case, i.e., when $\sigma=0$.
For simplicity, in this paper we assume that the strong convexity constant $\mu$, the Lipschitz constant $\tL$, and the variance $\sigma^2$ are the same among all local objective functions. With the introduction of $\tf_i$'s in \eqref{eq:tf}, we can study the cases of both strongly convex and general convex objective functions for the multi-agent optimization problem \eqref{eq:problem_multi_agent}. For convenience, we introduce the following notation for describing the overall objective function:
\begin{equation}
	\label{eq:ftf}
	f(x):=\tf(x) + \mu\nu(x)\text{ where }\tf(x):=\tsum_{i=1}^{m}\tf_i(x^{(i)})\text{ and }\nu(x):=\tsum_{i=1}^{m}\nu_i(x^{(i)}).
\end{equation}

\subsection{Related works and contributions of this paper}
In many real-world applications, the topology of the communication network may constantly change, due to possible connectivity issues especially for agents that are Internet-of-Things devices like cellphones or car sensors. 
Therefore, it is important to design decentralized algorithms with sampling or computation complexity independent of the graph topology \cite{nedic2018network}. 
Ideally, our goal is to develop decentralized methods whose sampling (resp. gradient computation) complexity bounds are graph invariant and in the same order as those of centralized methods for solving stochastic (resp. deterministic) problems.
Although there are fruitful research results in which the sampling or computational complexities and communication complexities are separated for decentralized (stochastic) optimization, for instance, (stochastic) (sub)gradient based algorithms \cite{tsitsiklis1986distributed,nedic2009distributed,duchi2011dual, tsianos2012consensus, nedic2014distributed, hong2017stochastic}, dual type or ADMM based decentralized methods \cite{terelius2011decentralized, boyd2011distributed,shi2014linear, aybat2017distributed, wei20131, chang2014multi}, communication efficient methods \cite{lan2020communication, chen2018lag,koloskova2019decentralized,liu2019communication}, and second-order methods \cite{mokhtari2016dqm, mokhtari2016decentralized}, 
none of the existing decentralized methods can achieve the aforementioned goal for problem \eqref{eq:problem_multi_agent}. To the best of our knowledge, only some very recently developed decentralized algorithms could match the complexities of those of centralized algorithms for the case when $X=\R^d$ \cite{li2020decentralized,kovalev2020optimal,li2020optimal,ye2020pmgt,alghunaim2020decentralized}. However, none of them would apply when we consider a general constraint feasible set $X$ in problem \eqref{eq:problem_multi_agent}.
In this paper, we pursue the goal of developing graph topology invariant decentralized optimization algorithms based on the gradient sliding methods~\cite{lan2016gradient,lan2016accelerated}. Our contributions in this paper can be summarized as follows.

First, for the general convex deterministic problem \eqref{eq:problem_multi_agent}, we propose a novel decentralized algorithm, namely the primal dual sliding (PDS) algorithm, that is able to compute an $\varepsilon$-solution with $\cO(1)((\tL/\varepsilon)^{1/2})$ gradient complexity. This complexity is invariant with respect to the topology of the communication network. To the best of our knowledge, this is the first decentralized algorithm for problem \eqref{eq:problem_multi_agent} that achieves the same order of gradient complexity bound as those for centralized methods. Such gradient complexity can be improved to $\cO(1)((\tL/\mu)^{1/2}\log(1/\varepsilon))$ for the strongly convex case of problem \eqref{eq:problem_multi_agent}.

Second, for the general convex stochastic problem \eqref{eq:problem_multi_agent} in which the gradients of the objective functions can only be estimated through stochastic first-order oracle, we propose a stochastic primal dual sliding (SPDS) algorithm that is able to compute an $\varepsilon$-solution with $\cO(1)((\tL/\varepsilon)^{1/2}+(\sigma/\varepsilon)^2)$ sampling complexity. The established complexity is also invariant with respect to the topology of the communication network. Such result can be improved to $\cO(1)((\tL/\mu)^{1/2}\log(1/\varepsilon) + \sigma^2/\varepsilon)$ for the strongly convex case.

Third, as a byproduct, we show that a simple extension of the PDS algorithm can be applied to solve certain convex-concave saddle point problems with bilinear coupling. For general convex smooth case, the number of gradient evaluations of $\nabla \tf$ and matrix operations (involving $\cA$ and $\cA^\top$) are bounded by $\cO(1)((\tL/\varepsilon)^{1/2})$ and $\cO(1)((\tL/\varepsilon)^{1/2}+\|\cA\|/\varepsilon)$ respectively. As a special case, our proposed algorithm is also able to solve convex smooth optimization problems with linear constraints. This is the first time such complexity results are achieved for linearly constrained convex smooth optimization. We also extend our results to strongly convex problems.

\subsection{Organization of the paper}
This paper is organized as follows. In section \ref{sec:PDS} we present the PDS algorithm for decentralized optimization. In section \ref{sec:SPDS} we present the SPDS algorithm for problems that require sampling of stochastic gradients. The byproduct results on general bilinearly coupled saddle point problems are presented in Section \ref{sec:SPP}. To facilitate reading, we postpone all the major proofs to Section \ref{sec:analysis}. In Section \ref{sec:numerical} we present some preliminary numerical experiment results. Finally the concluding remarks are presented in Section \ref{sec:conclusion}.

\section{The primal dual sliding algorithm for decentralized optimization}
\label{sec:PDS}

In this section, we propose a primal dual sliding (PDS) algorithm for solving the linearly constrained formulation \eqref{eq:problem_decen} for decentralized multi-agent optimization in which each agent $i$ has access to the deterministic first-order information of its local objective $f_i$ (i.e., $\sigma=0$ in the assumption \eqref{eq:assum_so}). 
Inspired by the saddle point formulation \eqref{eq:spp_decen} and the decoupling of the strongly convex term in \eqref{eq:tf},
we propose to study the following saddle point problem:
\begin{equation}
	\label{eq:SPP}
	\min_{x\in \cX}\max_{y\in \R^{md}, z\in \R^{md}} \mu\nu(x) + \langle x, y+\cA^\top z\rangle - \tf^*(y).
\end{equation}
Here 
$\tf^*$ is the convex conjugate
of the function $\tf$ defined in \eqref{eq:ftf}. 
Specifically, $\tf^*$
can be described as the sum of the convex conjugates $\tf_i^*$ of functions $\tf_i$, i.e.,
$
	\tf^*(y) = \tsum_{i=1}^{m}\tf^*_i(y^{(i)}),\text{ where }y:=(y^{(1)}, \ldots, y^{(m)})\in\R^{md}\text{ and }\  y^{(1)},\ldots, y^{(m)}\in\R^d.
$
Our proposed PDS algorithm is a primal-dual algorithm that constantly maintains and updates primal variables $x$ and dual variables $y$ and $z$.
Most importantly, its algorithmic scheme allows the skipping of computations of gradients $\nabla \tf$ from time to time, to which we refer as a ``sliding'' feature. 
We describe the proposed PDS algorithm in Algorithms \ref{alg:PDS_agent} and \ref{alg:PDS}, where Algorithm \ref{alg:PDS_agent} describes the implementation from agent $i$'s perspective and Algorithm \ref{alg:PDS} focuses on implementation over the whole communication network. Here in Algorithm \ref{alg:PDS_agent} we assume that $\cA = \cL\otimes I_d$ is defined by the graph Laplacian $\cL$ of the communication network; similar implementation can be described if $\cA = \cB^\top\otimes I_d$ is defined by an oriented incidence matrix $\cB$.

\capstartfalse
\begin{figure}
\begin{minipage}{.95\linewidth}
\begin{algorithm}[H]
	\caption{\label{alg:PDS_agent}The PDS algorithm for solving \eqref{eq:SPP}, from agent $i$'s perspective}
	\footnotesize
	\begin{algorithmic}
		\State Choose $x_0^{(i)}, \xl_0^{(i)}\in X^{(i)}$, and set $\hx_0^{(i)} = x_{-1}^{(i)} = x_0^{(i)}$, $y_0^{(i)}=\nabla \tf(\xl_0^{(i)})$, and $z_0^{(i)}=0$.
		.
		\For {$k=1,\ldots,N$}
		\State 
		Compute
		\begin{align}
			\label{eq:txk_agent}
			\tx_k^{(i)} = & x\kp^{(i)} + \lambda_k(\hx\kp^{(i)} - x_{k-2}^{(i)})
			\\
			\label{eq:xlk_agent}
			\xl_k^{(i)} = & (\tx_k^{(i)} + \tau_k\xl\kp^{(i)})/(1+\tau_k)
			\\
			\label{eq:yk_agent}
			y_k^{(i)} = & \nabla \tf_i(\xl_k^{(i)})
		\end{align}
		and set $x_k^{0,(i)} = x\kp^{(i)}$, $z_k^{0,(i)} = z\kp^{(i)}$, and $x_k^{-1,(i)} = x\kp^{T\kp-1,(i)}$ (set $x_1^{-1,(i)}=x_0^{(i)}$).
		\For {$t=1,\ldots,T_k$}
		\State 
		Compute
		\begin{align}
			\label{eq:tukt_agent}
			\tu_k^{t,(i)} = & x_k^{t-1,(i)} + \alpha_k^t(x_k^{t-1,(i)} - x_k^{t-2,(i)})
			\\
			\label{eq:zkt_agent}
			z_k^{t,(i)} =& \argmin_{z^{(i)}\in \R^d}\left\langle -\sum_{j\in N_i}\cL^{(i,j)}\tu_k^{t,(j)}, z^{(i)}\right\rangle + \frac{q_k^t}{2}\|z^{(i)} - z_k^{t-1,(i)}\|_2^2
			\\
			\label{eq:xkt_agent}
			x_k^{t,(i)} = & \argmin_{x^{(i)}\in X^{(i)}} \mu\nu_i(x) + \left\langle y_k^{(i)} + \sum_{j\in N_i}\cL^{(i,j)} z_k^{t,(j)}, x^{(i)}\right\rangle 
			\\
			&\qquad\qquad + \eta_k^tV_i(x_{k}^{t-1,(i)}, x^{(i)}) + p_kV_i(x\kp^{(i)}, x^{(i)})
		\end{align} 
		\EndFor
		\State 
		Set $x_k^{(i)} = x_k^{T_k,(i)}$, $z_k^{(i)} = z_k^{T_k,(i)}$, $\hx_k^{(i)} = \sum_{t=1}^{T_k}x_k^{t,(i)}/T_k$, and $\hz_k^{(i)} = \sum_{t=1}^{T_k}z_k^{t,(i)}/T_k$.
		\EndFor 
		\State Output $\xu_N^{(i)}:=\left(\sum_{k=1}^{N}\beta_k\right)^{-1}\sum_{k=1}^{N}\beta_k\hx_k^{(i)}$.		
	\end{algorithmic}
\end{algorithm}
\begin{algorithm}[H]
	\caption{\label{alg:PDS}The PDS algorithm for solving \eqref{eq:SPP}, whole network perspective}
	\footnotesize
	\begin{algorithmic}
		\State Choose $x_0, \xl_0\in X^{(i)}$, and set $\hx_0 = x_{-1} = x_0$, $y_0=\nabla \tf(\xl_0)$, and $z_0=0$.
		\For {$k=1,\ldots,N$}
		\State Compute
		\begin{align}
		\label{eq:txk}
			\tx_k = & x\kp + \lambda_k(\hx\kp - x_{k-2})
			\\
			\label{eq:yk}
			y_k = & \argmin_{y\in \R^{md}}\langle -\tx_k, y\rangle + \tf^*(y) + \tau_kW(y\kp, y).
		\end{align}
		\State Set $x_k^0 = x\kp$, $z_k^0 = z\kp$, and $x_k^{-1} = x\kp^{T\kp-1}$ (when $k=1$, set $x_1^{-1}=x_0$).
		\For {$t=1,\ldots,T_k$}
		\begin{align}
		\label{eq:tukt}
		\tu_k^t = & x_k\tp + \alpha_k^t(x_k\tp - x_k^{t-2})
		\\
		\label{eq:zkt}
		z_k^t =& \argmin_{z\in \R^{md}}\langle -\cA\tu_k^t, z\rangle + \frac{q_k^t}{2}\|z - z_k^{t-1}\|_2^2
		\\
		\label{eq:xkt}
		x_k^t = & \argmin_{x\in \cX} \mu\nu(x) + \langle y_k + \cA^\top z_k^t, x\rangle + \eta_k^tV(x_{k}^{t-1}, x) + p_kV(x\kp, x)
		\end{align} 
		\EndFor
		\State Set $x_k = x_k^{T_k}$, $z_k = z_k^{T_k}$, $\hx_k =\sum_{t=1}^{T_k}x_k^t/T_k$, and $\hz_k = \sum_{t=1}^{T_k}z_k^t/T_k$.
		\EndFor 
		\State Output $\xu_N:=\left(\sum_{k=1}^{N}\beta_k\right)^{-1}\sum_{k=1}^{N}\beta_k\hx_k$.
	\end{algorithmic}
\end{algorithm}
\end{minipage}
\end{figure}
\capstarttrue

A few remarks are in place for Algorithms \ref{alg:PDS_agent} and \ref{alg:PDS}. First, in Algorithm \ref{alg:PDS_agent}, $V_i$ and $W_i$ are prox-functions utilized by the $i$-th agent. Specifically, they are defined based on strongly convex functions $\nu$ and $\tf_i^*$ respectively:
\begin{align}
	\label{eq:Vi}
	V_i(\hat x^{(i)}, x^{(i)}) := & \nu(x^{(i)}) - \nu(\hat x^{(i)}) - \langle \nu'(\hat x^{(i)}), x^{(i)} - \hat x^{(i)}\rangle,\ \forall \hat x^{(i)}, x^{(i)}\in X^{(i)},
	\\
\label{eq:Wi}
W_i(\hat y^{(i)}, y^{(i)}):=&\tf_i^*(y^{(i)}) - \tf_i^*(\hat y^{(i)}) - \langle (\tf^*_i)'(\hat y^{(i)}), y^{(i)}- \hat y^{(i)}\rangle,\ \forall \hat y^{(i)}, y^{(i)}\in \R^{d},
\end{align}
where $(\tf^*_i)'$ denotes the subgradient of $\tf^*_i$. We will also use the notation
\begin{align}
	\label{eq:VW}
	V(\hat x, x):=\tsum_{i=1}^{m}V_i(\hat x^{(i)}, x^{(i)})\text{ and }W(\hat y, y):=\tsum_{i=1}^{m}W_i(\hat y^{(i)}, y^{(i)})
\end{align}
in the description of Algorithm \ref{alg:PDS}. 
Second, the major difference between Algorithms \ref{alg:PDS_agent} and \ref{alg:PDS} is the update step for $y_k$. However, it can be shown that the two update steps for $y_k$, i.e., \eqref{eq:yk_agent} and \eqref{eq:yk}, are equivalent and hence Algorithms \ref{alg:PDS_agent} and \ref{alg:PDS} are equivalent. Indeed, in view of the definition of $W$ in \eqref{eq:Wi} and \eqref{eq:VW} and the optimality condition of \eqref{eq:yk_SPP}, we have
$
	-\tx_k + (1+\tau_k)(\tf^*)'(y_k) - \tau_k(\tf^*)'(y\kp) = 0
$
for certain subgradients $(\tf^*)'(y_k)$ and $(\tf^*)'(y\kp)$ of $\tf^*$. Consequently, if $\xl\kp = (\tf^*)'(y\kp)$ or equivalently $y\kp = \nabla \tf(\xl\kp)$, and $\xl_k$ is defined recursively as in \eqref{eq:xlk_agent}, then we have $y_k = \nabla \tf(\xl_k)$. 
Noting in the description of Algorithms \ref{alg:PDS_agent} and \ref{alg:PDS} that $y_0=\nabla \tf(\xl_0)$, by induction we have that $y_k = \nabla \tf(\xl_k)$ for all $k$. The converse of the above derivation is also true, hence we can conclude that the descriptions of $y_k$ in \eqref{eq:yk_agent} and \eqref{eq:yk} are equivalent. It is then immediate to observe that Algorithms \ref{alg:PDS_agent} and \ref{alg:PDS} are equivalent descriptions of the proposed PDS algorithm. 
Third, observe that there are two loops, namely the inner and outer loops, in the PDS algorithm. 
On one hand, communications are performed only in the inner loops (for $t=1,\ldots,T_k$) and hence the PDS algorithm has in total $2\sum_{k=1}^{N}{T_k}$ rounds of communications. 
On the other hand, gradient computations are only performed in the outer loop (for $k=1,\ldots,N$) and therefore the PDS algorithm has in total $N$ gradient evaluations. 
Here, note that gradient computations and communications are not performed in a one-to-one fashion; in fact the PDS algorithm skips the computations of gradients $\nabla \tf$ frequently. 
The skipping of the gradient computations is the ``sliding'' feature of PDS that allows the gradient complexity to be independent of the graph topology. 
Finally, observe that our output is a weighted average of $\hx_k$ from $k=1$ to $N$, in which $\hx_k$ is the average of inner loop computation results $\{x_k^t\}_{t=1}^{T_k}$. Note that although we are using equal weights, $\hx_k$ may also be chosen as a weighted average of $x_k^t$'s. The analysis and complexity results will remain the same.

Our main convergence result for the PDS algorithm is described in Proposition \ref{pro:Qest} below. We also provide example parameter choices of Algorithm \ref{alg:PDS_agent} in Theorem \ref{thm:spp} that satisfy the conditions in Proposition \ref{pro:Qest}. To facilitate reading, we delay the proof of Proposition \ref{pro:Qest} to Section \ref{sec:analysis}.

\begin{pro}
	\label{pro:Qest}
	Suppose that the parameters of Algorithm \ref{alg:PDS_agent} satisfy the following conditions:
	\begin{itemize}
		\item For any $k\ge 2$, 
		\begin{align}
			\label{eq:cond_Qest_k}
			\begin{aligned}
			& \beta_k\tau_k\le \beta\kp(\tau\kp+1),\ \beta\kp = \beta_k\lambda_k,\ \tL\lambda_k\le p\kp\tau_k,
			\\
			& \beta_k T\kp\alpha_k^1= \beta\kp T_k,\ \alpha_k^1 \|\cA\|^2\le \eta\kp^{T\kp}q_k^{1},\ \beta_k T\kp q_k^1\le \beta\kp T_k q\kp^{T\kp},
			\\
			& \beta_k T\kp (\eta_k^1 + p_k T_k)\le \beta\kp T_k(\mu+\eta\kp^{T\kp}+p\kp);
		\end{aligned}
		\end{align}
		\item For any $t\ge 2$ and $k\ge 1$, 
		\begin{align}
			\label{eq:cond_Qest_t}
			\begin{aligned}
				& \alpha_k^t=1,\ \|\cA\|^2\le\eta_k\tp q_k^t, \ q_k^t \le q_k\tp,\ \eta_k^t\le \mu+\eta_k\tp+p_k;
			\end{aligned}
		\end{align}
		\item In the first and last outer iterations, 
		\begin{align}
			\label{eq:cond_Qest_N}
			\begin{aligned}
				& \tau_1 = 0,\ p_N(\tau_N+1)\ge \tL, \text{ and }
				\eta_N^{T_N}q_{N}^{T_N}\ge \|\cA\|^2.
			\end{aligned}
		\end{align}
	\end{itemize}
	Then we have
	\begin{align}
	& \begin{aligned}
		 f(\xu_N) - f(x^*)
		\le & \left(\tsum_{k=1}^{N}\beta_k\right)^{-1}
		\beta_1\left(\tfrac{\eta_1^1}{T_1} + p_1\right)V(x_0, x^*)
		\text{ and }
	\end{aligned}
\\
		& \begin{aligned}
			 \|\cA \xu_N\|_2 \le & \left(\tsum_{k=1}^{N}\beta_k\right)^{-1}
			 \beta_1
			 \left[
			 \tfrac{q_1^1}{2T_1}(\|z^*\|_2+1)^2 
			+ \left(\tfrac{\eta_1^1}{T_1} + p_1\right)V(x_0, x^*)\right].
		\end{aligned}
	\end{align}	
\end{pro}


\begin{thm}
	\label{thm:spp}
	Denote $\tau:=\sqrt{2\tL/\mu}$ and $\Delta:=\lceil 2\tau + 1\rceil$ if $\mu>0$, and $\Delta:=+\infty$ if $\mu=0$. Suppose that the parameters in Algorithm \ref{alg:PDS_agent} are set to the following: for all $k\le \Delta$,
	\begin{align}
		\label{eq:par_beforeDelta}
		\begin{aligned}
		\tau_k = \tfrac{k-1}{2},\ \lambda_k = \tfrac{k-1}{k},\ \beta_k = k,\ p_k = \tfrac{2\tL}{k},\ T_k = \left\lceil \tfrac{kR\|\cA\|}{\tL}\right\rceil,
		\end{aligned}
	\end{align}
	for all $k> \Delta$,
	\begin{align}
		\label{eq:par_afterDelta}
		\begin{aligned}
		\tau_k = \tau,\ \lambda_k = \lambda:=\tfrac{\tau}{1+\tau},\ \beta_k = \Delta\lambda^{-(k-\Delta)},\ p_k = \tfrac{\tL}{1+\tau},\ T_k = \left\lceil \tfrac{2(1+\tau)R\|\cA\|}{\tL\lambda^{\frac{k-\Delta}{2}}}\right\rceil,
		\end{aligned}
	\end{align}
	and for all $k$ and $t$,
	\begin{align}
		\label{eq:par_allk}
			& \eta_k^t = (p_k+\mu)(t-1) + p_kT_k,q_k^t = \tfrac{\tL T_k}{2 \beta_kR^2}, \alpha_k^t = \begin{cases}
	\tfrac{\beta\kp T_k}{\beta_kT\kp},&\text{ if }k\ge 2\ \text{ and } t=1;
	\\
	1, & \text{otherwise}.
\end{cases}
	\end{align}
	Applying Algorithm \ref{alg:PDS_agent} to solve problem \eqref{eq:problem_decen} we have
	\begin{align}
		\label{eq:fest_T}
		& f(\xu_N) - f(x^*) \le \min\left\{\tfrac{2}{N^2}, \lambda^{N-\Delta}\right\}\cdot 4\tL V(x_0,x^*)
		\\
		\label{eq:Axest_T}
		& \|\cA\xu_N\|_2 \le \min\left\{\tfrac{2}{N^2}, \lambda^{N-\Delta}\right\}\left[\tfrac{\tL}{4 R^2}(\|z^*\|_2+1)^2 + 4\tL V(x_0, x^*)\right].
	\end{align}
	Specially, 
	if we set
	\begin{align}
		\label{eq:R}
		R = \tfrac{\|z^*\|_2 +1}{4\sqrt{V(x_0,x^*)}},
	\end{align}
	then 
	we have the following gradient and communication complexity results for Algorithm \ref{alg:PDS_agent} (here $\cO(1)$ is a constant that is independent of $N$, $\|\cA\|$ and $\epsilon$):
	\begin{enumerate}[label=\emph{\alph*})]
		\item 
		\label{itm:result_nsc}
		If $\mu=0$, i.e., problem \eqref{eq:problem_decen} is smooth and convex, then 
		we can obtain an $\varepsilon$-solution after at most $N:=\left\lceil 4\sqrt{\tL V(x_0,x^*)/\varepsilon}\right\rceil$ gradient evaluations and $2N+96\|\cA\|/\varepsilon$ communications.
		\item 
		\label{itm:result_sc}
		If $\mu>0$, i.e., problem \eqref{eq:problem_decen} is smooth and strongly convex, then 
		we can obtain an $\varepsilon$-solution after at most $N:=
		\cO(1)\left(1 + \sqrt{\tL/\mu}\log(\tL V(x_0,x^*)/\varepsilon)\right)$ gradient evaluations and $2N + \cO(1)\|\cA\|(1+1/\sqrt{\varepsilon})$ communications.
	\end{enumerate}
\end{thm}
\begin{proof}
	In view of the selection of $\eta_k^t$ and $q_k^t$, we can observe that
	\begin{align}
		\eta_k\tp q_k^t\ge \tfrac{\tL p_k}{2 \beta_k}\left(\tfrac{T_k}{R}\right)^2 = \begin{cases}
			\|\cA\|^2\left(\frac{\tL T_k}{kR\|\cA\|}\right)^2, & k\le \Delta;
			\\
			\tfrac{(2\tau+2)\|\cA\|^2}{\Delta}\left(\frac{\tL\lambda^{\frac{k-\Delta}{2}}T_k}{2(1+\tau)R\|\cA\|}\right)^2, & k> \Delta,
		\end{cases}
	\end{align}
	and together with the definitions of $\Delta$ and $T_k$ we have $\eta_k\tp q_k^t \ge \|\cA\|^2$ for all $k\ge 1$ and $t\ge 2$. Using this observation, it is easy to verify that \eqref{eq:cond_Qest_t} holds. Noting that the above observation also implies that $\eta_N^{T_N}q_N^{T_N}\ge \|\cA\|^2$, and also noting that $\tau_1=0$ and that $p_N(\tau_N+1) =\tL({N+1})/{N}$ when $N\le \Delta$ and $p_N(\tau_N+1) = \tL$ when $ N> \Delta$,
	we have \eqref{eq:cond_Qest_N} holds. It remains to verify \eqref{eq:cond_Qest_k}.
    Note that for all $k$ we have $\beta_k T\kp \alpha_k^1 = \beta\kp T_k$ and $\beta_k T\kp q_k^1 = \tL T_kT\kp /(2R^2) = \beta\kp T_k q\kp^{T\kp}$.
	We consider three cases: $2\le k\le \Delta$,  $k= \Delta+1$, and $k> \Delta+1$. 
	When $2\le k\le\Delta$, we have
	\begin{align}
		& \beta\kp(\tau\kp+1) = {(k-1)k}/{2} = \beta_k\tau_k,
		\\
		& \beta\kp = k-1 = \beta_k\lambda_k, 
		\\
		& \tL\lambda_k = {\tL (k-1)}/{k} < \tL = p\kp\tau_k,
		\\
		& \tfrac{\eta\kp^{T\kp}q_k^{1}}{\alpha_k^1\|\cA\|^2} \ge \tfrac{\eta\kp^{1}q_k^{1}}{\alpha_k^1\|\cA\|^2} \ge 
		\left(\tfrac{\tL T\kp}{(k-1)R\|\cA\|}\right)^2\ge 1,
		\\
		& \beta_kT\kp(\eta_k^1 + p_kT_k) - \beta\kp T_k(\mu+\eta\kp^{T\kp}+p\kp) = - (k-1)\mu T_k T\kp \le 0.
	\end{align}
	When $k=\Delta+1$, from the definition of $\Delta$ we have $2\tau+1\le \Delta\le 2\tau+2$. Therefore
	\begin{align}
		& \beta_k\tau_k - \beta\kp(\tau\kp+1) = \Delta\lambda^{-1}\tau - {\Delta(\Delta+1)}/{2} = (\Delta/2)(2(1+\tau) - (\Delta + 1)) \le 0,
		\\
		& \beta\kp - \beta_k\lambda_k = \Delta - \Delta\lambda^{-1}\lambda = 0, 
		\\
		& \tL\lambda_k - p\kp\tau_k = \tL\lambda - (2\tL/\Delta)\tau \le \tL\lambda - (\tL/(1+\tau))\tau = 0,
		\\
		& \tfrac{\eta\kp^{T\kp}q_k^{1}}{\alpha_k^1\|\cA\|^2} \ge \tfrac{\eta\kp^{1}q_k^{1}}{\alpha_k^1\|\cA\|^2}
		\ge \left(\tfrac{\tL T_\Delta}{\Delta R\|\cA\|}\right)^2\ge 1,
		\\
		& \begin{aligned}
			\beta_k T\kp(\eta_k^1 + p_kT_k) - \beta\kp T_k(\mu+\eta\kp^{T\kp}+p\kp) 
			=&\mu\left[\Delta (\tau-1) - 2\tau^2\right] T_kT\kp\le 0.
		\end{aligned}
	\end{align}
	When $k> \Delta+1$, we have
	\begin{align}
		& \beta\kp(\tau\kp+1) = \Delta\lambda^{-(k-1-\Delta)}(\tau+1) = \Delta\lambda^{-(k-\Delta)}\tau = \beta_k\tau_k,
		\\
		& \beta\kp = \Delta\lambda^{-(k-1-\Delta)} = \beta_k\lambda_k, 
		\\
		& \tL\lambda_k = \tL\tau/(1+\tau) = p\kp\tau_k,
		\\
		& \tfrac{\eta\kp^{T\kp}q_k^{1}}{\alpha_k^1\|\cA\|^2} \ge \tfrac{\eta\kp^{1}q_k^{1}}{\alpha_k^1\|\cA\|^2} 
	    \ge \tfrac{2\tau+2}{\Delta}\cdot\left(\tfrac{\tL \lambda^{(k-1-\Delta)/2} T\kp}{2(1+\tau)R\|\cA\|}\right)^2\ge 1,
		\\
		& \begin{aligned}
			\tfrac{\beta_k T\kp(\eta_k^1 + p_kT_k)}{\beta\kp T_k(\mu+\eta\kp^{T\kp}+p\kp)} 
			= & \tfrac{\tau^2(1+\tau)}{\tau(\tau^2+1+\tau)}\le 1.
		\end{aligned}
	\end{align}
	We now can conclude that \eqref{eq:cond_Qest_k} holds. 
	Note that the summation $\sum_{k=1}^N\beta_k$ can be lower bounded as follows: when $N\le \Delta$, we have $\sum_{k=1}^N\beta_k = N(N+1)/2> N^2/2$; when $N>\Delta$, we have $\sum_{k=1}^N\beta_k>\beta_N \ge \lambda^{-(N-\Delta)}$.
	

	Applying Proposition \ref{pro:Qest} and using the above note on the summation $\sum_{k=1}^N\beta_k$ and the parameter selection in $\eqref{eq:par_beforeDelta}$, we obtain results \eqref{eq:fest_T} and \eqref{eq:Axest_T}. Moreover, setting $R$ to \eqref{eq:R} we conclude that
	\begin{align}
		\max\{f(\xu_N) - f(x^*), \|\cA\xu_N\|_2\}\le \min\left\{\tfrac{2}{N^2}, \lambda^{N-\Delta}\right\}8\tL V(x_0, x^*).
	\end{align}
	We conclude immediately that $\xu_N$ will be an $\varepsilon$-solution if 
	$
		\label{eq:Nest}
		N =\left\lceil 4\sqrt{{\tL V(x_0,x^*)}/{\varepsilon}}\right\rceil
	$ 
	or 
	$
	N=\left\lceil \Delta + \log_{\lambda^{-1}}({8\tL V(x_0,x^*)}/{\varepsilon})\right\rceil.
	$
	As a consequence, when the problem is smooth and convex, i.e., $\mu=0$, by setting $N=\left\lceil 4\sqrt{\tL V(x_0,x^*)/\varepsilon}\right\rceil$ we can obtain an $\varepsilon$-solution $\xu_N$. Using the definitions of $T_k$ and $R$ in \eqref{eq:par_beforeDelta} and \eqref{eq:R} respectively, the total number of communication rounds required by the PDS algorithm can be bounded by 
	$
		2\sum_{k=1}^{N}T_k \le 2N + N(N+1)R\|\cA\|/\tL \le 2N + 96\|\cA\|/\varepsilon.
	$
	When the problem is smooth and strongly convex, i.e., $\mu>0$, we may set $N:=\left\lceil \Delta + \log_{\lambda^{-1}}({8\tL V(x_0,x^*)}/{\varepsilon})\right\rceil$ to obtain an $\varepsilon$-solution $\xu_N$. In view of the fact that $-\log\lambda = \log(1+1/\tau) \ge 1/(1+\tau)$, the total number of gradient computations is bounded by
	\begin{align}
		N\le & 1 + \Delta + \frac{1}{-\log\lambda}\log\tfrac{8\tL V(x_0,x^*)}{\varepsilon} \le 1 + \Delta + (\tau+1)\left(\log\tfrac{8\tL V(x_0,x^*)}{\varepsilon}\right)
		\\
		= & \cO(1)\left(1 + \sqrt{\tL/\mu}\log(\tL V(x_0,x^*)/\varepsilon)\right).
	\end{align}
	Also, observing that
	$
		\lambda^{-N} \le {8\lambda^{-\Delta-1}\tL V(x_0,x^*)}/{\varepsilon},
	$
	using the definitions of $\lambda$, $T_k$ and $R$ in \eqref{eq:par_beforeDelta}, \eqref{eq:par_afterDelta} and \eqref{eq:R}, the total number of communication rounds required by the PDS algorithm is bounded by
	\begin{align}
		2\tsum_{k=1}^{N}T_k \le & 2\tsum_{k=1}^{\Delta}\left(\tfrac{kR\|\cA\|}{\tL}+1\right) + 2\tsum_{k=\Delta+1}^{N}\left(\tfrac{2(1+\tau)\|\cA\|R}{\tL\lambda^{\tfrac{k-\Delta}{2}}}+1\right) 
		\\
		= & 2N + \tfrac{\Delta(\Delta+1)R\|\cA\|}{\tL} + \tfrac{4(1+\tau)R\lambda^{\frac{\Delta}{2}}}{\tL(1 - \lambda^{\frac{1}{2}})}\cdot\|\cA\|(\lambda^{-\frac{N}{2}} - \lambda^{-\frac{\Delta}{2}})
		\\
		\le & 2N + \cO(1)\|\cA\|+\cO(1)\|\cA\|/\sqrt{\varepsilon}.
	\end{align}

\end{proof}

\vgap 

In view of the above theorem, we can conclude that the number of gradient evaluations required by the PDS algorithm to obtain an $\varepsilon$-solution to the decentralized optimization problem \eqref{eq:problem_decen} does not depend on the communication network topology; it is only dependent on the Lipschitz constant $\tL$,  strong convexity parameter $\mu$, and proximity to an optimal solution $V(x_0,x^*)$. 
Moreover, the number of gradient evaluations required by the PDS algorithm is in the same order of magnitude as those optimal centralized algorithms (see, e.g.,  \cite{nesterov2004introductory}). To the best of our knowledge, this is the first time such complexity bounds for gradient evaluations are established for solving decentralized optimization problems over general constraint sets.
Note that the communications complexity for computing an $\varepsilon$-solution still depends on the graph topology of the communication network.

\section{The stochastic primal-dual sliding algorithm for stochastic decentralized optimization}
\label{sec:SPDS}

In this section, we follow the assumption of the unbiased gradient estimators stated in \eqref{eq:assum_so} and propose a stochastic primal-dual sliding (SPDS) algorithm for solving the decentralized multi-agent optimization problem \eqref{eq:problem_decen} under the stochastic first-order oracle. 
Our proposed SPDS algorithm is a stochastic primal-dual algorithm that 
possesses a similar ``sliding'' scheme as the PDS algorithm. In particular, SPDS integrates the mini-batch technique and uses increasing batch size, while skipping stochastic gradient computation from time to time. We describe the SPDS algorithm in Algorithms \ref{alg:SPDS_agent} and \ref{alg:SPDS}.

\capstartfalse
\begin{figure}[h]
	\begin{minipage}{.95\linewidth}
		\begin{algorithm}[H]
			\caption{\label{alg:SPDS_agent}The SPDS algorithm for solving \eqref{eq:SPP}, from agent $i$'s perspective}
			\footnotesize
			\begin{algorithmic}
				\State Modify \eqref{eq:yk_agent} and \eqref{eq:xkt_agent} in Algorithm \ref{alg:PDS_agent} to
				\begin{align}
					\label{eq:vki}
					v_k^{(i)} = & \tfrac{1}{c_k}\tsum_{j=1}^{c_k}G_i(\xl_k^{(i)}, \xi_k^{(i),j})\text{ and }
					\\
					\label{eq:xkt_agent_S}
					x_k^{t,(i)} = & \argmin_{x^{(i)}\in X^{(i)}} \mu\nu_i(x) + \left\langle v_k^{(i)} + \tsum_{j\in N_i}\cL^{(i,j)} z_k^{t,(j)}, x^{(i)}\right\rangle 
				\\
				&\qquad\qquad + \eta_k^tV_i(x_{k}^{t-1,(i)}, x^{(i)}) + \tfrac{p_k}{2}\|x\kp^{(i)} - x^{(i)}\|_2^2
				\end{align}
			\end{algorithmic}
		\end{algorithm}
	
		\vspace{-.8cm}
		
		\begin{algorithm}[H]
			\caption{\label{alg:SPDS}The SPDS algorithm for solving \eqref{eq:SPP}, whole network perspective}
			\footnotesize
			\begin{algorithmic}
				\State Modify \eqref{eq:xkt} in Algorithm \ref{alg:PDS} to
				\begin{align}
					\label{eq:xkt_S}
					x_k^t = & \argmin_{x\in \cX} \mu\nu(x) + \langle v_k + \cA^\top z_k^t, x\rangle + \eta_k^tV(x_{k}^{t-1}, x) + \tfrac{p_k}{2}\|x\kp - x\|_2^2,
			\end{align} 
			in which $v_k$ satisfies $\E[v_k] = y_k$ and 
			\begin{align}
				\label{eq:sigmak_S}
				\E[\|v_k - y_k\|_2^2] \le \sigma_k^2,\text{ where }\sigma_k:=\tfrac{\sigma}{\sqrt{c_k}}
			\end{align}
			\end{algorithmic}
		\end{algorithm}
	\end{minipage}
\end{figure}
\capstarttrue

The difference between the SPDS and PDS algorithms is that the SPDS algorithm uses a mini-batch of $c_k$ stochastic samples of gradients while PDS requires the access of exact gradients, i.e., $y_k=\nabla \tf(\xl_k)$. 
As a result, the update of $y_k$ is replaced by a mini-batch update step and equations \eqref{eq:xkt_agent} and \eqref{eq:xkt} of the PDS algorithm are replaced by their stochastic version accordingly. 
In order to use the mini-batch technique, we need to replace the prox-function $V(\cdot,\cdot)$ in the PDS algorithm in Algorithm \ref{alg:PDS} by Euclidean distances. 
Recalling the assumption on the variance of stochastic samples in \eqref{eq:assum_so} and noting that the batch size at the $k$-th outer iteration is $c_k$, it is easy to obtain that the variance of the gradient estimator $v_k$ is $\sigma_k^2 := \sigma^2/c_k$, as stated in \eqref{eq:sigmak_S}. 
It should also be noted that in many applications the $i$-th agent may be able to collect samples $\{\xi_k^{(i),j}\}_{j=1}^{c_k}$ of the underlying random variable in \eqref{eq:vki} for computing the gradient estimator $v_k^{(i)}$. For example, in online supervised machine learning the $i$-th agent can receive samples $\xi_k^{(i),j}$'s from online data streams of the training dataset. In such case, the SPDS algorithm allows all agents to collect stochastic samples and perform communications simultaneously. This is because that as soon as the mini-batch computation of the gradient estimator $v_k^{(i)}$ in \eqref{eq:vki} is finished in the $k$-th outer iteration, the $i$-th agent can start collecting future stochastic samples $\{\xi\kn^{(i),j}\}_{j=1}^{c\kn}$ for the next outer iteration. Consequently, the collection of future samples and the communication with neighboring agents in the current $k$-th outer iteration can be performed simultaneously.

The convergence of the SPDS algorithm is described in the following proposition,
whose proof is presented in Section \ref{sec:analysis}. Followed by the proposition we also describe example parameter choices in Theorem \ref{thm:spp_S}.

\begin{pro}
	\label{pro:Qest_S}
	Suppose that the parameters of Algorithm \ref{alg:SPDS_agent} satisfy conditions described in \eqref{eq:cond_Qest_k}--\eqref{eq:cond_Qest_N}, among which two conditions are modified to
	\begin{align}
		\label{eq:cond_modified_S}
		2\tL\lambda_k\le p\kp\tau_k,\ \forall k\ge 2,\text{ and }p_N(\tau_N+1)\ge 2\tL,
	\end{align}
	in which the Lipschitz constant $\tL$ is with respect to the Euclidean norm. We have
	\begin{align}
		\label{eq:fest_S}
		& \begin{aligned}
			& \E\left[f(\xu_N) - f(x^*)\right]\le  \left(\tsum_{k=1}^{N}\beta_k\right)^{-1}\left[\tfrac{\beta_1}{2}\left(\tfrac{\eta_1^1}{T_1} + p_1\right)\|x_0 - x^*\|_2^2	+ \tsum_{k=1}^{N}\tfrac{\beta_k \sigma^2}{p_kc_k}\right]\text{ and }
		\end{aligned}
	\\
		\label{eq:Axbest_S}
		& \begin{aligned}
			\E\left[\|\cA \xu_N\|_2\right] \le & \left(\tsum_{k=1}^{N}\beta_k\right)^{-1}\left[\tfrac{\beta_1q_1^1}{2T_1}(\|z^*\|_2+1)^2 \right.
			\\
			& \qquad\qquad\qquad\left.+ \tfrac{\beta_1}{2}\left(\tfrac{\eta_1^1}{T_1} + p_1\right)\|x_0 - x^*\|_2^2	+ \tsum_{k=1}^{N}\frac{\beta_k \sigma^2}{p_kc_k}\right].
		\end{aligned}
	\end{align}	
\end{pro}




\begin{thm}
	\label{thm:spp_S}
		Denote $\tau:=2\sqrt{\tL/\mu}$ and $\Delta:=\lceil 2\tau + 1\rceil$ if $\mu>0$, and $\Delta:=+\infty$ if $\mu=0$. Suppose that the maximum number of outer iterations $N$ is specified, and that the parameters in Algorithm \ref{alg:SPDS_agent} are set to the following: for all $k\le \Delta$,
	\begin{align}
		\label{eq:par_beforeDelta_S}
		\begin{aligned}
			\tau_k = \tfrac{k-1}{2},\ \lambda_k = \tfrac{k-1}{k},\ \beta_k = k,\ p_k = \tfrac{4\tL}{k}, T_k = \left\lceil \tfrac{kR\|\cA\|}{\tL}\right\rceil, c_k = \left\lceil \tfrac{\min\{N,\Delta\} \beta_kc}{p_k\tL}\right\rceil,
		\end{aligned}
	\end{align}
	for all $k\ge \Delta+1$,
	\begin{align}
		\label{eq:par_afterDelta_S}
		\begin{aligned}
			& \tau_k = \tau,\ \lambda_k = \lambda:=\tfrac{\tau}{1+\tau},\ \beta_k = \Delta\lambda^{-(k-\Delta)},\ p_k = \tfrac{2\tL}{1+\tau},\ 
			\\
			& T_k = \left\lceil \tfrac{2(1+\tau)R\|\cA\|}{\tL\lambda^{\tfrac{k-\Delta}{2}}}\right\rceil, c_k = \left\lceil\tfrac{(1+\tau)^{2}\Delta c}{\tL^2\lambda^{\frac{k+N-2\Delta}{2}}}\right\rceil,
		\end{aligned}
	\end{align}
	and for all $k$ and $t$,
	\begin{align}
		\label{eq:par_allk_S}
		& \eta_k^t = (p_k+\mu)(t-1) + p_kT_k, \ q_k^t = \tfrac{\tL T_k}{4 \beta_kR^2}, \
    \alpha_k^t = \begin{cases}
			\tfrac{\beta\kp T_k}{\beta_kT\kp}&k\ge 2\ \text{ and } t=1
			\\
			1 & \text{otherwise}.
		\end{cases}
	\end{align}
	Applying Algorithm \ref{alg:SPDS_agent} to solve problem \eqref{eq:problem_decen} we have
	\begin{align}
		\label{eq:fest_T_S}
		& f(\xu_N) - f(x^*) \le \min\left\{\tfrac{2}{N^2}, \lambda^{N-\Delta}\right\}\left[ 4\tL \|x_0 - x^*\|_2^2 + \tfrac{\tL\sigma^2}{c}\right]\text{ and }
		\\
		\label{eq:Axest_T_S}
		& \|\cA\xu_N\|_2 \le \min\left\{\tfrac{2}{N^2}, \lambda^{N-\Delta}\right\}\left[\tfrac{\tL}{8 R^2}(\|z^*\|_2+1)^2 + 4\tL \|x_0 - x^*\|_2^2 + \tfrac{\tL\sigma^2}{c}\right].
	\end{align}
	Specially, 
	if we set
	\begin{align}
		\label{eq:Rc_k}
		\begin{aligned}
		& R = \tfrac{\|z^*\|_2 +1}{4\|x_0-x^*\|_2}\text{ and } c= \tfrac{2\sigma^2}{\|x_0 -x^*\|_2^2},
		\end{aligned}
	\end{align}
	then 
	 we have the following sampling and communication complexity results for Algorithm \ref{alg:SPDS_agent} (here $\cO(1)$ is a constant that is independent of $\sigma$, $N$, $\|\cA\|$ and $\epsilon$):
	\begin{enumerate}[label=\emph{\alph*})]
		\item 
		\label{itm:result_nsc_S}
		If $\mu=0$, 
		by setting $N=\left\lceil \sqrt{{10\tL \|x_0 - x^*\|_2^2}/{\varepsilon}}\right\rceil$  
		we can obtain a stochastic $\varepsilon$-solution with $N + 800\sigma^2 \|x_0 - x^*\|_2^2/\varepsilon$  gradient samples and $N + 600\|\cA\|/\varepsilon$ rounds of communication.
		\item 
		\label{itm:result_sc_S}
		If $\mu>0$, 
		by setting $N=\left\lceil \Delta + \log_{\lambda^{-1}}({5\tL \|x_0 - x^*\|_2^2}/{\varepsilon})\right\rceil$
		we can obtain a stochastic $\varepsilon$-solution with $\cO(1)\left(1 + \sqrt{\tL/\mu}\log(\tL \|x_0 - x^*\|_2^2/\varepsilon) + \sigma^2(1+1/\varepsilon)\right)$ samples and $\cO(1)\left(1 + \sqrt{\tL/\mu}\log(\tL \|x_0 - x^*\|_2^2/\varepsilon) + \|\cA\|(1+1/\sqrt{\varepsilon})\right)$ rounds of communication.
	\end{enumerate}
\end{thm}

\begin{proof}
    Following a similar argument to the proof of Theorem \ref{thm:spp}, it is easy to verify that the parameter selections in \eqref{eq:par_beforeDelta_S} - \eqref{eq:par_allk_S} satisfy the conditions stated in Proposition~\ref{pro:Qest_S}. The results in \eqref{eq:fest_T_S} and \eqref{eq:Axest_T_S} follows immediately from the proposition and our parameter selections.
	We now consider the case when $\mu=0$ (hence $\Delta=+\infty$). 
	By the choice of $N$, the selections of $\beta_k$, $p_k$, $R$, and $c$ in \eqref{eq:par_beforeDelta_S} and \eqref{eq:Rc_k}, and noting that $\sum_{k=1}^{N}k^2\le N^3$ we can bound the total number of gradient samples by
	\begin{align}
		\tsum_{k=1}^{N}c_k \le & \tsum_{k=1}^{N}\left(1 + \tfrac{cNk^2}{4\tL^2}\right) = N + \tfrac{\sigma^2N}{2\tL^2 \|x_0 - x^*\|_2^2}\sum_{k=1}^{N}k^2
		\\
		\le & N + \tfrac{\sigma^2N^4}{2\tL^2 \|x_0 - x^*\|_2^2} = N+
\tfrac{800\sigma^2\|x_0-x^*\|_2^2}{\varepsilon^2}
		.
	\end{align}
In the case when $\mu>0$, similar to the proof of Theorem \ref{thm:spp} we can observe two bounds
$
	N\le \cO(1)\left(1 + \sqrt{\tL/\mu}\log(\tL \|x_0-x^*\|_2^2/\varepsilon)\right)
$
and
$
\lambda^{-N} \le (1/{\varepsilon})\cO(1).
$
In view of these bounds, the descriptions of $c_k$ and $c$ in \eqref{eq:par_beforeDelta_S}, \eqref{eq:par_afterDelta_S}, and \eqref{eq:Rc_k}, we can bound the total number of gradient samples by
\begin{align}
	\tsum_{k=1}^{N}c_k \le & \tsum_{k=1}^{\Delta}\left(1 + \frac{\Delta \sigma^2 k^2}{2\tL^2 \|x_0-x^*\|_2^2}\right) + \sum_{k=\Delta+1}^{N}\left(1 + \frac{(1+\tau)^{2}\Delta c}{2\tL^2\lambda^{\frac{N-\Delta}{2}}}\lambda^{-\frac{k-\Delta}{2}}\right)
	\\
	= & N + \cO(1)\sigma^2  + \cO(1)\frac{1-\lambda^{\frac{N-\Delta}{2}}}{1 - \lambda^{\frac{1}{2}}}\sigma^2\lambda^{-(N-\Delta)}
	\\
	\le & \cO(1)\left(1 + \sqrt{\tL/\mu}\log(\tL \|x_0-x^*\|_2^2/\varepsilon)\right) + \cO(1)\sigma^2 + \cO(1){\sigma^2}/{\varepsilon}.
\end{align}

\end{proof}
 
\vgap

In view of Theorem~\ref{thm:spp_S}, we can conclude that the sampling complexity achieved by the SPDS algorithm does not depend on the communication network topology; it is only dependent on the Lipschitz constant $\tL$,  strong convexity parameter $\mu$, the variance of the gradient estimators $\sigma^2$, and proximity to an optimal solution $\|x_0-x^*\|_2$. 
Moreover, the achieved sampling complexity is in the same order of magnitude as those optimal centralized algorithms, such as accelerated stochastic approximation method \cite{ghadimi2012optimal,ghadimi2013optimal}.
To the best of our knowledge, this is the first time such sampling complexity bounds are established for solving stochastic decentralized optimization problems of form \eqref{eq:problem_multi_agent}.
Indeed, the total number of communications required to compute an $\varepsilon$-solution still depends on the graph topology. It should also be noted that the constant factors in part a) of the above theorem (800 and 600) can be further reduced, e.g., by assuming that $N\ge 5$. However, our focus is on the development of order-optimal complexity results and hence we skip these refinements.

\section{A byproduct on solving bilinearly coupled saddle point problems}
\label{sec:SPP} 

In this section, we present a byproduct of our analysis and show that the PDS algorithm described in Algorithm \ref{alg:PDS} is in fact an efficient algorithm for solving certain convex-concave saddle point problem with linear coupling. Specifically, in this section we consider the following saddle point problem:
\begin{align}
	\label{eq:SPP_general}
	\min_{x\in \cX}\max_{z\in\cZ}f(x) + \langle \cA x, z\rangle - h(z)
\end{align}
where $\cX$ and $\cZ$ are closed convex sets, $\cA$ is a linear operator, and $f$ and $h$ are convex functions. Here, we assume that $h(z)$ is relatively simple, and that $f(x) = \tf(x) + \mu \nu(x)$, where $\mu\ge 0$, $\nu:\cX\to\R$ is a strongly convex function with strong convexity modulus $1$ with respect to a norm $\|\cdot\|$, and $\tf:X\to\R$ is a convex smooth function 
whose gradient is Lipschitz with constant $\tL$ 
with respect to the norm $\|\cdot\|$. Note that problem \eqref{eq:SPP_general} is equivalent to 
\begin{align}
	\label{eq:SPP_general_with_conjugate}
	\min_{x\in \cX}\max_{y\in\cY, z\in \cZ} \mu\nu(x) + \langle x, y+ \cA^\top z\rangle - \tf^*(y) - h(z)
\end{align}
where $\tf^*$ is the convex conjugate of $\tf$ and $\cY:=\dom \tf^*$.
Also, observe that when $\cZ$ is a vector space and $h(z) = \langle b, z\rangle$, the saddle point problem \eqref{eq:SPP_general} becomes 
\begin{align}
	\label{eq:problem_lco}
	\min_{x\in\cX} f(x) \st \cA x = b.
\end{align}
If in addition 
$b=0$, $\cZ=\R^{md}$, and $\cX$, $f$, and $\cA$ are defined in \eqref{eq:problem_decen}, then we obtain the linearly constrained reformulation of the multi-agent optimization problem \eqref{eq:problem_multi_agent}. 

We describe an extension of the PDS algorithm in Algorithm \ref{alg:PDS_SPP}. In the algorithm description, $U$, $V$, and $W$ are prox-functions defined by
\begin{align}
	\label{eq:U}
	U(\hat z, z):= & \zeta(z) - \zeta(\hz) - \langle \zeta'(\hz), z - \hz\rangle,
	\\
	\label{eq:V}
	V(\hx, x):= & \nu(x) - \nu(\hx) - \langle \nu'(\hx), x - \hx\rangle,\text{ and }
	\\
	\label{eq:W}
	W(v,y):=&\tf^*(y) - \tf^*(v) - \langle (\tf^*)'(v), y - v\rangle.
\end{align}
respectively, 
where $\nu$ and $\tf^*$ are defined in \eqref{eq:SPP_general_with_conjugate} and $\zeta:\cZ\to\R$ is a strongly convex function with strong convexity modulus $1$ with respect a norm $|\cdot|$. Noting from the strong convexity of $\zeta$ and $\nu$ we have for any $\hat z, z\in \cZ$ and $\hat x, x\in \cX$ that
\begin{align}
	\label{eq:UVbound}
	U(\hat z, z)\ge \tfrac{1}{2}|\hz - z|^2 \text{ and }V(\hx, x)\ge \tfrac{1}{2}\|\hx - x\|^2.
\end{align}
Moreover, since $\tf$ is convex and $\nabla \tf$ is Lipschitz with constant $\tL$ with respect to norm $\|\cdot\|$, $\tf^*$ should be strongly convex with strong convexity parameter $1/\tL$ with respect to the dual norm $\|\cdot\|_*$. Consequently, noting the definition of $W$ in \eqref{eq:W} we have
\begin{align}
	\label{eq:Wbound}
	& W(v, y)\ge \tfrac{1}{2\tL}\|v - y\|_*^2, \ \forall v, y\in \cY.
\end{align}

\begin{algorithm}[h]
	\caption{\label{alg:PDS_SPP}An equivalent alternative description of the PDS algorithm}
	\begin{algorithmic}
		\State Choose $x_0\in \cX$, $y_0\in \cY$, and $z_0\in \cZ$. Set $\hx_0 = x_{-1} = x_0$.
		\For {$k=1,\ldots,N$}
		\State Compute
		\begin{align}
			\label{eq:txk_SPP}
			\tx_k = & x\kp + \lambda_k(\hx\kp - x_{k-2}),
			\\
			\label{eq:yk_SPP}
			y_k := & \argmin_{y\in \cY}\langle -\tx_k, y\rangle + \tf^*(y) + \tau_kW(y\kp, y).
		\end{align}
		\State Set $x_k^0 = x\kp$, $z_k^0 = z\kp$, and $x_k^{-1} = x\kp^{T\kp-1}$ (when $k=1$, set $x_1^{-1}=x_0$).
		\For {$t=1,\ldots,T_k$}
		\begin{align}
			\label{eq:tukt_SPP}
			\tu_k^t = & x_k\tp + \alpha_k^t(x_k\tp - x_k^{t-2})
			\\
			\label{eq:zkt_SPP}
			z_k^t =& \argmin_{z\in \cZ}h(z) + \langle -\cA\tu_k^t, z\rangle + q_k^tU(z_k^{t-1}, z)
			\\
			\label{eq:xkt_SPP}
			x_k^t = & \argmin_{x\in \cX} \mu\nu(x) + \langle y_k + \cA^\top z_k^t, x\rangle + \eta_k^tV(x_{k}^{t-1}, x) + p_kV(x\kp, x)
		\end{align} 
		\EndFor
		\State Set $x_k = x_k^{T_k}$, $z_k = z_k^{T_k}$, $\hx_k = \sum_{t=1}^{T_k}x_k^t/T_k$, and $\hz_k = \sum_{t=1}^{T_k}z_k^t/T_k$.
		\EndFor 
		\State Output $\wu_N:=(\xu_N, \yu_N, \zu_N):=\left(\sum_{k=1}^{N}\beta_k\right)^{-1}\left(\sum_{k=1}^N\beta_k(\hx_k,y_k,\hz_k)\right)$.
	\end{algorithmic}
\end{algorithm}

Extending the analysis of Algorithm \ref{alg:PDS}, it is straightforward to observe that the equivalent description of the PDS algorithm described in Algorithm \ref{alg:PDS_SPP} can be applied to solve the more general saddle point problem \eqref{eq:SPP_general_with_conjugate}. We will evaluate the accuracy through the 
gap function defined by
\begin{align}
	\label{eq:Qhwk_SPP}
	\begin{aligned}
		Q(\hw_k,w) := & \left[\mu\nu(\hx_k) + \langle \hx_k, y+\cA^\top  z\rangle - \tf^*(y) - h(z)\right] 
		\\
		& - \left[\mu\nu(x) + \langle x, y_k+\cA^\top \hz_k\rangle - \tf^*(y_k) - h(\hz_k)\right],
	\end{aligned}
\end{align}
in which $\hw_k:=(\hx_k,y_k,\hz_k)$ and $w:=(x,y,z)$.
We have the following proposition describing the convergence result, whose proof will be presented in Section \ref{sec:analysis}.

\begin{pro}
	\label{pro:Qest_proof}
	Suppose that the conditions \eqref{eq:cond_Qest_k}--\eqref{eq:cond_Qest_N} (where $\|\cA\|$ is the norm induced by $\|\cdot\|$ and $|\cdot|$) described in Proposition \ref{pro:Qest} hold for Algorithm \ref{alg:PDS_SPP}. 
	For all $w\in \cX\times\cY\times \cZ$ we have the following convergence result when Applying Algorithm \ref{alg:PDS_SPP} to solve problem \eqref{eq:SPP_general_with_conjugate}:
	\begin{align}
		\label{eq:Qest_proof}
		\begin{aligned}
			&  Q(\wu_N, w) \le \left(\tsum_{k=1}^{N}\beta_k\right)^{-1}\beta_1\left[\tfrac{q_1^1}{T_1}U(z_0, z) + \left(\tfrac{\eta_1^1}{T_1} + p_1\right)V(x_0, x)\right].
		\end{aligned}
	\end{align}
	Moreover, in the special case of problem \eqref{eq:problem_lco} (i.e., if $\cZ$ is a vector space with Euclidean norm $|\cdot|=\|\cdot\|_2$, the prox-function $U(\cdot,\cdot)=\|\cdot-\cdot\|_2^2/2$, and $h(z) = \langle b, z\rangle$), then we can set the initial value $z_0=0$ in Algorithm \ref{alg:PDS_SPP} and obtain 
	\begin{align}
		\label{eq:fest_proof}
		& f(\xu_N) - f(x^*) \le \left(\tsum_{k=1}^{N}\beta_k\right)^{-1}\beta_1\left(\tfrac{\eta_1^1}{T_1} + p_1\right)V(x_0, x^*)\text{ and }
			\\
		\label{eq:Axbest_proof}
		& \|\cA \xu_N - b\|_2 \le \left(\tsum_{k=1}^{N}\beta_k\right)^{-1}\beta_1\left[ \tfrac{q_1^1}{2T_1} (\|z^*\|_2+1)^2 + \left(\tfrac{\eta_1^1}{T_1} + p_1\right)V(x_0, x^*)\right].
	\end{align}
\end{pro}

\vgap

We provide a set of example parameters for Algorithm \ref{alg:PDS_SPP} in the theorem below. The proof of the theorem below is similar to that of Theorem \ref{thm:spp} and thus is skipped.

\vgap

\begin{thm}
	\label{thm:spp_general}
	Suppose that the parameters of Algorithm \ref{alg:PDS_SPP} are set to the ones described in Theorem \ref{thm:spp}. Applying Algorithm \ref{alg:PDS_SPP} to solve problem \eqref{eq:SPP_general} we have
		\begin{align}
			\begin{aligned}
				  Q(\wu_N, w) 
				\le &  \min\left\{\tfrac{2}{N^2}, \lambda^{N-\Delta}\right\}\left[\tfrac{\tL}{2 R^2}U(z_0,z^*) + 4\tL V(x_0, x^*)\right].
			\end{aligned}
		\end{align}
		Specially, if $R$ is set to \eqref{eq:R}, then 
		we have the following complexity results for Algorithm \ref{alg:PDS_SPP} to obtain an $\epsilon$-solution of problem~\eqref{eq:SPP_general}
		(here $\cO(1)$ is a constant that is independent of $N$, $\|\cA\|$ and $\epsilon$):
		\begin{enumerate}[label=\emph{\alph*})]
			\item 
			If $\mu=0$, then the PDS algorithm can obtain an $\epsilon$-solution of problem~\eqref{eq:SPP_general}
			after at most $\cO(1)\left(\sqrt{{\tL V(x_0,x^*)}/{\varepsilon}}\right)$ number of gradient evaluations and $\cO(1)\left(\sqrt{{\tL V(x_0,x^*)}/{\varepsilon}}+{\|\cA\|}/{\varepsilon}\right)$ number of operator computations (involving $\cA$ and $\cA^\top$).
			\item 
			If $\mu>0$, then the PDS algorithm can obtain an $\epsilon$-solution of problem~\eqref{eq:SPP_general}
			after at most $\cO(1)\left(1 + \sqrt{2\tL/\mu}\log(\tL V(x_0,x^*)/\varepsilon\right)$ number of gradient evaluations and  $\cO(1)\left(1 + \sqrt{2\tL/\mu}\log(\tL V(x_0,x^*)/\varepsilon) + \|\cA\|(1+1/\sqrt{\varepsilon)}\right)$ number of operator computations (involving $\cA$ and $\cA^\top$).
		\end{enumerate}
\end{thm}

\vgap

To the best of our knowledge, this is the first time that a primal-dual type algorithm achieves the above complexity results for solving bilinearly coupled saddle point problems. The only other method that achieves such complexity results is presented in \cite{lan2016accelerated}, which requires to utilize a smoothing technique developed in \cite{nesterov2005smooth}. However, the smoothing technique requires that either set $\cX$ or $\cZ$ is compact and hence is not applicable to the constrained optimization problem \eqref{eq:problem_lco}. Our proposed algorithm is the first algorithm in the literature that is able to solve the linearly constrained optimization problem \eqref{eq:problem_lco} with the above complexities.

\section{Convergence analysis}
\label{sec:analysis}
Our goal in this section is to complete the remaining proofs in Sections \ref{sec:PDS}--\ref{sec:SPP}, including the proofs of Propositions \ref{pro:Qest}, \ref{pro:Qest_S}, and \ref{pro:Qest_proof}. 
We first need the following proposition that describes an estimate of the gap function in \eqref{eq:Qhwk_SPP}:
\begin{pro}
	\label{pro:Qest_raw}
	Suppose that $\hx_k=\sum_{t=1}^{T_k}x_k^t/T_k$ and $\hz_k=\sum_{t=1}^{T_k}z_k^t/T_k$, where the iterates $\{x_k^t\}_{t=1}^{T_k}$ and $\{z_k^t\}_{t=1}^{T_k}$ are defined by 
	\begin{align}
	\label{eq:zkt_proof}
	z_k^t =& \argmin_{z\in \cZ}h(z) + \langle -\cA\tu_k^t, z\rangle + q_k^tU(z_k^{t-1}, z)\text{ and }
	\\
	\label{eq:xkt_proof}
	x_k^t = & \argmin_{x\in \cX} \mu\nu(x) + \langle v_k + \cA^\top z_k^t, x\rangle + \eta_k^tV(x_{k}^{t-1}, x) + p_kV(x\kp, x)
	\end{align}
	respectively. Letting $\hw_k:=(\hx_k,y_k,\hz_k)$ we have
	\begin{align}
	\label{eq:Qest_raw}
	\tsum_{k=1}^{N}\beta_kQ(\hat w_k, w) + A + B \le C + D,\ \forall w:=(x,y,z)\in\cX\times\cY\times\cZ
	\end{align}
	where
	\begingroup
	\allowdisplaybreaks
	\begin{align}
	\label{eq:A}
	& \begin{aligned}
	A: = & \tsum_{k=1}^{N}\beta_k\left[- \langle \hx_k, y\rangle + \langle x, y_k\rangle + \langle v_k, \hx_k - x\rangle - \tf^*(v_k) + \tf^*(y) 
	\right.
	\\
	&\qquad\qquad\left. + \tfrac{p_k}{T_k}\tsum_{t=1}^{T_k}V(x\kp, x_k^t)\right],
	\end{aligned}
	\\
	\label{eq:B}
	& \begin{aligned}
	B:= & \tsum_{k=1}^{N}\tfrac{\beta_k}{T_k}\sum_{t=1}^{T_k}\left[\langle \cA^\top z_k^t - \cA^\top z, x_k^t - \tu_k^t\rangle  \right.
	\\
	& \qquad\qquad\qquad\left.+ q_k^tU(z_k\tp, z_k^t)  + \eta_k^tV(x_k\tp, x_k^t)\right],
	\end{aligned}
	\\
	\label{eq:C}
	& C:= \tsum_{k=1}^{N}\tfrac{\beta_k}{T_k}\sum_{t=1}^{T_k}\left[q_k^tU(z_k\tp, z) - q_k^tU(z_k^t,z)\right],
	\\
	\label{eq:D}
	& D:= \tsum_{k=1}^{N}\tfrac{\beta_k}{T_k}\sum_{t=1}^{T_k}\left[\eta_k^t V(x_k\tp, x) - (\mu+\eta_k^t+p_k)V(x_k^t, x) + p_kV(x\kp, x)\right].
	\end{align}
	\endgroup
\end{pro}

\begin{proof}
	By the optimality conditions of \eqref{eq:zkt_proof} and \eqref{eq:xkt_proof} and definitions of $U$ and $V$ in \eqref{eq:U} and \eqref{eq:V} respectively, we have 
	\begin{align}
		& \langle h'(z_k^t) - \cA \tu_k^t + q_k^t\zeta'(z_k^t) - q_k^t\zeta'(z_k\tp), z_k^t -z \rangle\le 0,\ \forall z\in \cZ,\text{ and }
		\\
		& \langle (\mu+\eta_k^t +p_k)\nu'(x_k^t) + v_k + \cA^\top z_k^t - \eta_k^t \nu'(x_k\tp) - p_k\nu'(x\kp), x_k^t - x\rangle\le 0, \ \forall x\in\cX.
	\end{align}
	In view of the above results, the convexity of $h$ and $\nu$, and the definitions of $U$ and $V$ we obtain the following two relations: for any $z\in\cZ$ and $x\in\cX$, 
	\begin{align}
	\label{eq:oc_zkt}
	& h(z_k^t) - h(z) + \langle -\cA\tu_k^t, z_k^t - z\rangle + q_k^tU(z_k\tp, z_k^t) + q_k^tU(z_k^t,z) \le q_k^tU(z_k\tp, z), \text{ and }
	\\
	& 
	\begin{aligned}
	& \langle v_k+\cA^\top z_k^t, x_k^t - x\rangle + \mu\nu(x_k^t) - \mu\nu(x) 
	\\
	& + \eta_k^tV(x_k\tp, x_k^t) + (\mu + \eta_k^t+p_k)V(x_k^t, x)  + p_kV(x\kp, x_k^t)
	\\
	\le & \eta_k^t V(x_k\tp, x) + p_kV(x\kp, x).
		\end{aligned}
	\label{eq:oc_xkt}
	\end{align}
	Summing up the two relations above, while noting that
	\begin{align}
	& \langle -\cA\tu_k^t, z_k^t - z\rangle + \langle v_k+\cA^\top z_k^t, x_k^t - x\rangle 
	\\
	= & \langle \cA^\top z_k^t - \cA^\top z, x_k^t - \tu_k^t\rangle + \langle \cA^\top z, x_k^t\rangle - \langle v_k+\cA^\top z_k^t, x\rangle + \langle v_k, x_k^t\rangle,
	\end{align}
	we have
	\begin{align}
	& \langle \cA^\top z_k^t - \cA^\top z, x_k^t - \tu_k^t\rangle + \langle \cA^\top z, x_k^t\rangle - \langle v_k+\cA^\top z_k^t, x\rangle + \langle v_k, x_k^t\rangle 
	\\
	& + h(z_k^t) - h(z) +  q_k^tU(z_k\tp, z_k^t) + q_k^tU(z_k^t,z) 
	\\
	& + \mu\nu(x_k^t) - \mu\nu(x) + \eta_k^tV(x_k\tp, x_k^t) + (\mu + \eta_k^t+p_k)V(x_k^t, x)  + p_kV(x\kp, x_k^t)
	\\
	\le & q_k^tU(z_k\tp, z) + \eta_k^t V(x_k\tp, x) + p_kV(x\kp, x).
	\end{align}
	Summing from $t=1,\ldots,T_k$ and noting the definitions of $\hx_k$ and $\hz_k$ and the convexity of functions $h$ and $\nu$ we have
	\begin{align}
	& \ T_k \left[\langle \cA^\top z, \hx_k\rangle - \langle v_k+\cA^\top \hz_k, x\rangle + \langle v_k, \hx_k\rangle + h(\hz_k) - h(z) + \mu\nu(\hx_k) - \mu\nu(x)\right]
	\\
	& \ + \tsum_{t=1}^{T_k}p_k \left[V(x\kp, x_k^t) + \langle \cA^\top z_k^t - \cA^\top z, x_k^t - \tu_k^t\rangle + q_k^tU(z_k\tp, z_k^t)  + \eta_k^tV(x_k\tp, x_k^t)\right]
	\\
	& \le  \tsum_{t=1}^{T_k}\left[q_k^tU(z_k\tp, z) - q_k^tU(z_k^t,z)\right. 
	\\
	& \qquad\qquad \left. + \eta_k^t V(x_k\tp, x) - (\mu + \eta_k^t+p_k)V(x_k^t, x) + p_kV(x\kp, x)\right].
	\end{align}
	We conclude \eqref{eq:Qest_raw} by multiplying the above relation by $\beta_k/T_k$, noting the definition of the gap function $Q(\hw_k,w)$ in \eqref{eq:Qhwk_SPP} and summing the resulting relation from $k=1, \dots, N$.
\end{proof}

\vgap

It should be noted that the above result is valid for any choices of $v_k$, $\tu_k^t$ and $x_{k-1}$ used in the descriptions of $z_k^t$ and $x_k^t$ in \eqref{eq:zkt_proof} and \eqref{eq:xkt_proof} respectively. 
Now we are ready to establish tight bounds of $A$, $B$, $C$, and $D$ in \eqref{eq:Qest_raw} through a series of technical lemmas. 
To provide a tight estimate of $A$, we start with the case when $v_k$ is the exact solution to the dual prox-mapping subproblem \eqref{eq:yk_SPP}.

\vgap

\begin{lem}
	\label{lem:Aest}
	Suppose that $v_k=y_k$ where $y_k$ is defined by \eqref{eq:txk_SPP} and \eqref{eq:yk_SPP}. If 
	\begin{align}
	\label{eq:cond_Aest}
	\tau_1=0, \beta_k\tau_k\le \beta\kp(\tau\kp+1),
	\beta\kp = \beta_k\lambda_k, \text{ and } \tL\lambda_k\le p\kp\tau_k,\ \forall k\ge 2,
	\end{align}
	then we have
	\begin{align}
	\label{eq:Aest}
	A \ge - \beta_{N} \langle x_{N-1} - \hx_N, y_N - y\rangle + \tfrac{\beta_Np_N}{2}\|x_{N-1} - \hx_{N}\|^2  + \beta_N(\tau_N+1)W(y_N,y).
	\end{align}
\end{lem}
\begin{proof}
	By the optimality condition of \eqref{eq:yk_SPP} and the definition of $W$ in \eqref{eq:W}
	\begin{align}
		\langle -\tx_k + (1+\tau_k)(\tf^*)'(y_k) -\tau_k(\tf^*)'(y\kp), y_k - y\rangle\le 0, \ \forall y\in\cY, 
	\end{align}
	which together with the convexity of $\tf^*$ and the definition of $W$ in \eqref{eq:W} implies
		\begin{align}
			\label{eq:oc_yk}
			\begin{aligned}
			& \langle -\tx_k, y_k - y\rangle + \tf^*(y_k) - \tf^*(y) + \tau_kW(y\kp, y_k) + (\tau_k+1)W(y_k,y) 
			\\
			\le & \tau_kW(y\kp, y), \ \forall y\in \cY.
		\end{aligned}
		\end{align}
	Combining the above relation with $A$ in \eqref{eq:A} and noting that $v_k=y_k$ we have
	\begin{align}
	A \ge & \tsum_{k=1}^{N}\beta_k
	\left[
	\vphantom{\sum_{t=1}^{T_k}}
	(\tau_k+1)W(y_k,y) - \tau_kW(y\kp, y) + \tau_kW(y\kp, y_k)  
	\right.
	\\
	& \qquad\qquad 
	\left. - \langle \hx_k, y\rangle + \langle y_k, \hx_k\rangle + \langle -\tx_k, y_k - y\rangle + \tfrac{p_k}{T_k}\tsum_{t=1}^{T_k}V(x\kp, x_k^t)\right]
	\end{align}
	From the definition of $\tx_k$ we can observe that
	\begin{align}
		& -\langle \hx_k, y\rangle + \langle y_k, \hx_k\rangle + \langle -\tx_k, y_k - y\rangle = -\langle x\kp + \lambda_k(\hx\kp - x_{k-2}) - \hx_k, y_k - y\rangle
		\\
		= & \lambda_k\langle x_{k-2} - \hx\kp, y\kp - y\rangle - \langle x\kp - \hx_k, y_k - y\rangle + \lambda_k\langle x_{k-2} - \hx\kp, y_k - y\kp\rangle.
	\end{align}
	Also, by the definition of $\hx_k$ and the fact that $V$ is lower bounded in \eqref{eq:UVbound} we have
	\begin{align}
		& \tfrac{1}{T_k}\tsum_{t=1}^{T_k}p_kV(x\kp, x_k^t) \ge p_kV(x\kp, \hx_k)\ge \frac{p_k}{2}\|x\kp - \hx_k\|^2.
	\end{align}
	Applying the above two observations, the bound of $W$ in \eqref{eq:Wbound}, the parameter assumption \eqref{eq:cond_Aest}, and recalling that $x_{-1}=\hx_0$ in Algorithm \ref{alg:PDS_SPP} we have
	\begin{align}
	A \ge & \tsum_{k=1}^{N}\left[\beta_k(\tau_k+1)W(y_k,y) - \beta_k\tau_kW(y\kp, y)\right. 
	\\
	& \qquad + \beta_k\lambda_k\langle x_{k-2} - \hx\kp, y\kp - y\rangle - \beta_k\langle x\kp - \hx_k, y_k - y\rangle 
	\\
	& \qquad \left.+ \beta_k\lambda_k\langle x_{k-2} - \hx\kp, y_k - y\kp\rangle + \tfrac{\beta_k p_k}{2}\|x\kp - \hx_k\|^2 + \tfrac{\beta_k\tau_k}{2\tL}\|y\kp - y_k\|_*^2\right]
	\\
	\ge & \beta_N(\tau_N+1)W(y_N,y) - \beta_{N} \langle x_{N-1} - \hx_N, y_N - y\rangle + \tfrac{\beta_Np_N}{2}\|x_{N-1} - \hx_{N}\|^2
	\\
	& + \tsum_{k=2}^{N}\left[\beta_k\lambda_k\|x_{k-2} - \hx\kp\|\|y_k - y\kp\|_*\right.\\
	& \qquad 
	+ \left.\tfrac{\beta\kp p\kp}{2}\|x_{k-2} - \hx\kp\|^2 + \tfrac{\beta_k\tau_k}{2\tL}\|y\kp - y_k\|_*^2\right]
	\\
	\ge & \beta_N(\tau_N+1)W(y_N,y) - \beta_{N} \langle x_{N-1} - \hx_N, y_N - y\rangle + \tfrac{\beta_Np_N}{2}\|x_{N-1} - \hx_{N}\|^2,
	\end{align}
	where the last inequality follows from the simple relation $a^2 + b^2\ge 2ab$ and
	\begin{align}
	(\beta_k\lambda_k)^2 - \beta\kp p\kp({\beta_k\tau_k}/{\tL}) = ({\beta_k^2\lambda_k}/{\tL})(\tL\lambda_k - p\kp \tau_k) \le 0.
	\end{align}
\end{proof}


There are cases when $v_k$ is not equal to $y_k$, for example, in Algorithm~\ref{alg:SPDS}. The following lemma provide a tight bound for $A$ under this case.


\begin{lem}
	\label{lem:Aest_S}
	Suppose that $y_k$ is defined by \eqref{eq:txk_SPP} and \eqref{eq:yk_SPP} and $v_k = v_k(y_k,\xi_k)$ is an unbiased estimator of $y_k$ with respect to random variable $\xi_k$ such that $\E[\|v_k - y_k\|_*]=0$ and $\E[\|v_k - y_k\|_*^2]\le \sigma_k^2$, where $\|\cdot\|_*$ is the dual norm of $\|\cdot\|$. If the parameters satisfy 
	\begin{align}
		\label{eq:cond_Aest_S}
		\tau_1=0, \beta_k\tau_k\le \beta\kp(\tau\kp+1),
		\beta\kp = \beta_k\lambda_k, \text{ and } 2\tL\lambda_k\le p\kp\tau_k,\ \forall k\ge 2,
	\end{align}
	then we have
	\begin{align}
		\label{eq:Aest_S}
		\E[A] \ge & \E\left[- \beta_{N} \langle x_{N-1} - \hx_N, y_N - y\rangle + \tfrac{\beta_Np_N}{4}\|x_{N-1} - \hx_N\|^2 - \tsum_{k=1}^{N}\frac{\beta_k \sigma_k^2}{p_k}
		\right.
		\\
		& \qquad+ \beta_N(\tau_N+1)W(y_N,y) 
		\left.\vphantom{\tsum_{t=1}^{T_k}}\right].
	\end{align}
\end{lem}
\begin{proof}
	In view of the optimality condition described previously in \eqref{eq:oc_yk} and the definition of $A$ in \eqref{eq:A} we have
	$
		A \ge A_1 + A_2
	$
	where
	\begin{align}
		A_1 := & \tsum_{k=1}^{N}\beta_k\left[- \langle \hx_k, y\rangle + \langle y_k, \hx_k\rangle + \langle -\tx_k, y_k - y\rangle + \tfrac{p_k}{2T_k}\tsum_{t=1}^{T_k}V(x\kp, x_k^t)\right.
		\\
		& \qquad\qquad \left.\vphantom{\tsum_{t=1}^{T_k}}
		+ \tau_kW(y\kp, y_k) + (\tau_k+1)W(y_k,y) - \tau_kW(y\kp, y)\right]\text{ and }
		\\
		A_2 := & \tsum_{k=1}^{N}\tfrac{\beta_k}{T_k} \tsum_{t=1}^{T_k}\left[ \tfrac{p_k}{2}V(x\kp, x_k^t) + \langle v_k - y_k, x_k^t - x\kp\rangle + \langle v_k - y_k, x\kp - x\rangle\right].
	\end{align}
	Applying the same argument as in Lemma \ref{lem:Aest} to $A_1$, we have 
	\begin{align}
		A_1 \ge & - \beta_{N} \langle x_{N-1} - \hx_N, y_N - y\rangle + \tfrac{\beta_Np_N}{4}\|x_{N-1} - \hx_{N}\|^2  + \beta_N(\tau_N+1)W(y_N,y) .
	\end{align}
	To finish the proof it suffices to show that $\E\left[A_2\right] \ge -\sum_{k=1}^{N}{\beta_k \sigma_k^2}/{p_k}$. Noting from the definition of $v_k$ that
	$
	\E_{\xi_1,\ldots,\xi_{k-1}}\left[\langle v_k - y_k, x\kp - x\rangle\right] = 0
	$ and applying the bound of $V$ in \eqref{eq:UVbound} and Cauchy-Schwartz inequality to the description of $A_2$ we have
	\begin{align}
		\E\left[A_2\right] 
		\ge & \E\left[\tsum_{k=1}^{N}\tfrac{\beta_k}{T_k} \sum_{t=1}^{T_k} \left(\frac{p_k}{4}\|x\kp - x_k^t\|^2 + \| v_k - y_k\|_* \|x_k^t - x\kp\|\right)\right]
		\\
		\ge & \E\left[\tsum_{k=1}^{N}\tfrac{\beta_k}{T_k} \sum_{t=1}^{T_k}-\frac{1}{p_k}\|v_k - y_k\|_*^2\right] = -\tsum_{k=1}^{N}\tfrac{\beta_k \sigma_k^2}{p_k}.
	\end{align}
\end{proof}

\vgap 

We continue with the estimates of $B$, $C$, and $D$ in the following two lemmas.

\vgap 

\begin{lem}
	\label{lem:Best}
	Suppose that
	\begin{align}
	\label{eq:cond_Best_t}
	& \alpha_k^t=1,\ \|\cA\|^2\le\eta_k\tp q_k^t,\ \forall t\ge 2, k\ge 1,\text{ and }
	\\
	\label{eq:cond_Best_k}
	& \beta_k T\kp\alpha_k^1= \beta\kp T_k,\ \alpha_k^1 \|\cA\|^2\le \eta\kp^{T\kp}q_k^{1},\ \forall k\ge 2,
	\end{align}
	where $\|\cA\|$ is the norm induced by norms $\|\cdot\|$ and $|\cdot|$ introduced in \eqref{eq:UVbound}. Then 
	\begin{align}
	\label{eq:Best}
	B \ge & \tfrac{\beta_N}{T_N}\|\cA\| |z_{N} - z|\| x_N - x_N^{T_N-1}\| + \tfrac{\beta_N\eta_N^{T_N}}{2T_N}\|x_N^{T_N-1} - x_N^{T_N}\|^2.
	\end{align}
\end{lem}
\begin{proof}
	By the definition of $\tu_k^t$ in \eqref{eq:tukt} and bounds of $U$ and $V$ in \eqref{eq:UVbound} we have
	\begin{align}
	& \tsum_{t=1}^{T_k}\left[\langle \cA^\top z_k^t - \cA^\top z, x_k^t - \tu_k^t\rangle + q_k^tU(z_k\tp, z_k^t)  + \eta_k^tV(x_k\tp, x_k^t)\right]
	\\
	\ge & 
	\tsum_{t=1}^{T_k}\left[- \alpha_k^t \langle \cA^\top z_k\tp - \cA^\top z, x_k\tp - x_k^{t-2}\rangle + \langle \cA^\top z_k^t - \cA^\top z, x_k^t - x_k\tp\rangle \right.
	\\
	&\quad \left.+ \tfrac{q_k^t}{2}\|z_k\tp - z_k^t\|^2 + \tfrac{\eta_k^t}{2}\|x_k\tp - x_k^t\|^2 - \alpha_k^t \langle \cA^\top z_k^t - \cA^\top z_k\tp, x_k\tp - x_k^{t-2}\rangle\right].
	\end{align}
	Also observe that $\alpha_k^t=1$ for all $t\ge 2$ we have
	\begin{align}
		& \tsum_{t=1}^{T_k}- \alpha_k^t \langle \cA^\top z_k\tp - \cA^\top z, x_k\tp - x_k^{t-2}\rangle + \langle \cA^\top z_k^t - \cA^\top z, x_k^t - x_k\tp\rangle 
		\\
		= & -\alpha_k^1\langle \cA^\top z_k^0 - \cA^\top z, x_k^0 - x_k^{-1}\rangle + \langle \cA^\top z_{k}^{T_k} - \cA^\top z, x_k^{T_k} - x_k^{T_k-1}\rangle.
	\end{align}
	In addition, in view of the assumptions \eqref{eq:cond_Best_t} and \eqref{eq:cond_Best_k} we have
	\begin{align}
		& \tsum_{t=1}^{T_k}\left[\frac{q_k^t}{2}\|z_k\tp - z_k^t\|^2 + \frac{\eta_k^t}{2}\|x_k\tp - x_k^t\|^2 - \alpha_k^t \langle \cA^\top z_k^t - \cA^\top z_k\tp, x_k\tp - x_k^{t-2}\rangle\right]
		\\
		=&  \tfrac{q_k^1}{2}\|z_k^0 - z_k^1\|^2 + \tfrac{\eta_k^{T_k}}{2}\|x_k^{T_k-1} - x_k^{T_k}\|^2 - \alpha_k^1 \langle z_k^1 - z_k^0, \cA(x_k^0 - x_k^{-1})\rangle
	\\
	& + \tsum_{t=2}^{T_k} \left[\frac{q_k^t}{2}\|z_k\tp - z_k^t\|^2 + \frac{\eta_k\tp}{2}\|x_k^{t-2} - x_k\tp\|^2 - \langle z_k^t - z_k\tp, \cA(x_k\tp - x_k^{t-2})\rangle\right]
	\\
	\ge & \tfrac{q_k^1}{2}\|z_k^0 - z_k^1\|^2 + \tfrac{\eta_k^{T_k}}{2}\|x_k^{T_k-1} - x_k^{T_k}\|^2 - \alpha_k^1 \|\cA\| |z_k^1 - z_k^0|\| x_k^0 - x_k^{-1}\|.
	\end{align}
	Applying the above three relations to the definition of $B$ in \eqref{eq:B}, recalling the definitions of $x_k^0$, $z_k^0$, $x_k^{-1}$ in Algorithm \ref{alg:PDS_SPP}, and using the assumption \eqref{eq:cond_Best_k} we have
	\begin{align}
	 B\ge & \tsum_{k=1}^N\frac{\beta_k}{T_k}
	\left[\vphantom{\frac{\eta_k^{T_k}}{2}}
	-\alpha_k^1\langle \cA^\top z_k^0 - \cA^\top z, x_k^0 - x_k^{-1}\rangle + \langle \cA^\top z\kn^{0} - \cA^\top z, x\kn^{0} - x\kn^{-1}\rangle
	\right.
	\\
	& \qquad\qquad\left.
	+ \tfrac{q_k^1}{2}\|z_k^0 - z_k^1\|^2 + \tfrac{\eta_k^{T_k}}{2}\|x_k^{T_k-1} - x_k^{T_k}\|^2 - \alpha_k^1 \|\cA\| |z_k^1 - z_k^0|\| x_k^0 - x_k^{-1}\| 
	\right]
	\\
	= & \tfrac{\beta_N}{T_N}\langle \cA^\top z_{N} - \cA^\top z, x_N - x_N^{T_N-1}\rangle
	\\
	& + \tfrac{\beta_1q_1^1}{2T_k}\|z_1^0 - z_1^1\|^2 + \tfrac{\beta_N\eta_N^{T_N}}{2T_N}\|x_N^{T_N-1} - x_N^{T_N}\|^2
	\\
	& + \sum_{k=2}^N\left[ \tfrac{\beta_kq_k^1}{2T_k}\|z_k^0 - z_k^1\|^2 + \tfrac{\beta\kp\eta\kp^{T\kp}}{2T\kp}\|x_k^{-1} - x_k^{0}\|^2 -\tfrac{\beta_k\alpha_k^1}{T_k} \|\cA\| |z_k^1 - z_k^0|\| x_k^0 - x_k^{-1}\|\right]
	\\
	\ge & \tfrac{\beta_N}{T_N}\|\cA\| |z_{N} - z|\| x_N - x_N^{T_N-1}\| + \tfrac{\beta_N\eta_N^{T_N}}{2T_N}\|x_N^{T_N-1} - x_N^{T_N}\|^2.
	\end{align}
	Here in the last inequality we use the following result from the assumption \eqref{eq:cond_Best_k}:
	\begin{align}
	& (\beta_k\alpha_k^1\|\cA\|/T_k)^2 - (\beta_kq_k^1/T_k)(\beta\kp\eta\kp^{T\kp} /T\kp) 
	\\
	= & (\beta_k/T_k)^2\alpha_k^1(\alpha_k^1\|\cA\|^2 - q_k^1\eta\kp^{T\kp})\ge 0.
	\end{align}
\end{proof}

\vgap

\begin{lem}
	\label{lem:CDest}
	If
	\begin{align}
	\label{eq:cond_Cest}
	q_k^t \le q_k\tp,\ \forall t\ge 2, k\ge 1\text{ and }\beta_k T\kp q_k^1\le \beta\kp T_k q\kp^{T\kp},\ \forall k\ge 2,
	\end{align}
	then $C$ in \eqref{eq:C} can be bounded by
	\begin{align}
	\label{eq:Cest}
	\begin{aligned}
	C \le & \tfrac{\beta_1q_1^1}{T_1}U(z_0, z) - \tfrac{\beta_Nq_N^{T_N}}{T_N}U(z_N,z).
	\end{aligned}
	\end{align}
	Also, if 	
	\begin{align}
	\label{eq:cond_Dest}
	\begin{aligned}
	& \eta_k^t\le \mu+\eta_k\tp+p_k,\ \forall t\ge 2, k\ge 1 \text{ and }
	\\
	& \beta_k T\kp (\eta_k^1 + p_k T_k)\le \beta\kp T_k(\mu+\eta\kp^{T\kp}+p\kp),\ \forall k\ge 2,
	\end{aligned}
	\end{align}
	then $D$ in \eqref{eq:C} can be bounded by
	\begin{align}
	\label{eq:Dest}
	\begin{aligned}
	D \le & \tfrac{\beta_1}{T_1}(\eta_1^1 + p_1T_1)V(x_0, x) - \tfrac{\beta_N}{T_N}(\mu+\eta_N^{T_N}+p_N)V(x_N,x).
	\end{aligned}
	\end{align}
\end{lem}

\begin{proof}
	In view of \eqref{eq:cond_Cest} and the fact that $z_k^0 = z\kp$ and $z_k = z_k^{T_k}$, we have:
	\begin{align}
	C \le & \tsum_{k=1}^{N}\left[\frac{\beta_k q_k^1}{T_k}U(z\kp, z) - \frac{\beta_k q_k^{T_k}}{T_k}U(z_k,z)\right] \le \frac{\beta_1q_1^1}{T_1}U(z_0, z) - \frac{\beta_Nq_N^{T_N}}{T_N}U(z_N,z).
	\end{align}
	Similarly, by \eqref{eq:cond_Dest} and noting that $x_k^0 = x\kp$ and $x_k = x_k^{T_k}$ we have
	\begin{align}
	D \le & \tsum_{k=1}^{N}\tfrac{\beta_k}{T_k}\left[\eta_k^1 V(x_k^0, x) - (\mu+\eta_k^{T_k}+p_k)V(x_k^{T_k}, x) + p_kT_kV(x\kp, x)\right]
	\\
	= & \tsum_{k=1}^{N}\left[\frac{\beta_k}{T_k}(\eta_k^1 + p_kT_k) V(x\kp, x) - \frac{\beta_k}{T_k}(\mu+\eta_k^{T_k}+p_k)V(x_k, x)\right]
	\\
	\le & \tfrac{\beta_1}{T_1}(\eta_1^1 + p_1T_1)V(x_0, x) - \tfrac{\beta_N}{T_N}(\mu+\eta_N^{T_N}+p_N)V(x_N,x).
	\end{align}
\end{proof}

\vgap

With Lemmas \ref{lem:Aest}, \ref{lem:Best}, and \ref{lem:CDest} and Proposition \ref{pro:Qest_raw}, we are ready to prove Proposition \ref{pro:Qest_proof}. 

 \vgap

\begin{proof}[Proof of Proposition \ref{pro:Qest_proof}]
	Note that the conditions \eqref{eq:cond_Qest_k} and \eqref{eq:cond_Qest_t} for parameters in Algorithm \ref{alg:PDS_SPP} are exactly the parameter assumptions \eqref{eq:cond_Aest}, \eqref{eq:cond_Best_t} -- \eqref{eq:cond_Dest}, 
	in Lemmas \ref{lem:Aest}, \ref{lem:Best}, and \ref{lem:CDest}.
	In view of the definition of $\wu_N$, the convexity of $Q$, Proposition \ref{pro:Qest_raw} and Lemmas \ref{lem:Aest}, \ref{lem:Best}, and \ref{lem:CDest}, we have
	\begin{align}
		& \left(\tsum_{k=1}^{N}\beta_k\right)Q(\wu_N, w) \le \tsum_{k=1}^{N}\beta_kQ(\hat w_k, w)
		\\
		\le &  \beta_{N} \langle x_{N-1} - \hx_N, y_N - y\rangle - \tfrac{\beta_Np_N}{2}\|x_{N-1} - \hx_{N}\|^2  - \beta_N(\tau_N+1)W(y_N,y)
		\\
		& - \tfrac{\beta_N}{T_N}\|\cA\| |z_{N} - z|\| x_N - x_N^{T_N-1}\| - \tfrac{\beta_N\eta_N^{T_N}}{2T_N}\|x_N^{T_N-1} - x_N^{T_N}\|^2
		\\
		& + \tfrac{\beta_1q_1^1}{T_1}U(z_0, z) - \tfrac{\beta_Nq_N^{T_N}}{T_N}U(z_N,z)
		\\
		& + \tfrac{\beta_1}{T_1}(\eta_1^1 + p_1T_1)V(x_0, x) - \tfrac{\beta_N}{T_N}(\mu+\eta_N^{T_N}+p_N)V(x_N,x)
		.
	\end{align}	
	We conclude \eqref{eq:Qest_proof} from the above inequality by noting from \eqref{eq:cond_Qest_N}, \eqref{eq:UVbound}, and \eqref{eq:Wbound} that
	\begin{align}
		& \begin{aligned}
		& \beta_{N} \langle x_{N-1} - \hx_N, y_N - y\rangle - \tfrac{\beta_Np_N}{2}\|x_{N-1} - \hx_{N}\|^2  - \beta_N(\tau_N+1)W(y_N,y)
		\\
		\le & \tfrac{2\beta_N}{p_N}\|y_N - y\|^2 - \tfrac{\beta_N(\tau_N+1)}{2\tL}\|y_N - y\|^2 \le 0, \text{ and }
		\\
		\end{aligned}
	\\
	& \begin{aligned}
		& - \tfrac{\beta_N}{T_N}\|\cA\| |z_{N} - z|\| x_N - x_N^{T_N-1}\| - \tfrac{\beta_N\eta_N^{T_N}}{2T_N}\|x_N^{T_N-1} - x_N^{T_N}\|^2 - \tfrac{\beta_Nq_N^{T_N}}{T_N}U(z_N,z)
		\\
		\le & \tfrac{\beta_Nq_N^{T_N}}{2T_N}|z_N - z|^2 - \tfrac{\beta_Nq_N^{T_N}}{2T_N}|z_N - z|^2
		= 0.
	\end{aligned}
	\end{align}

	In the special case when $\cZ$ is a vector space with Euclidean norm $|\cdot|=\|\cdot\|_2$, the prox-function $U(\cdot,\cdot)=\|\cdot-\cdot\|_2^2/2$, and $h(z) = \langle b, z\rangle$, let us choose any saddle point $w^*=(x^*,y^*,z^*)$ that solves the SPP \eqref{eq:SPP_general_with_conjugate} and denote
	\begin{align}
		\label{eq:wustar}
		\wu^*_N:=(x^*, y^*, \zu_N^*)\text{ where }\zu_N^*:=\tfrac{\|z^*\|_2+1}{\|\cA\xu_N-b\|_2}(A\xu_N-b).
	\end{align}
	Note that in such special case the SPP \eqref{eq:SPP_general_with_conjugate} reduces to the linear constrained optimization \eqref{eq:problem_lco}, in which $x^*$ is an optimal solution, $y^*=\nabla \tf(x^*)$, and $z^*$ is the optimal Lagrange multiplier associated with the equality constraints. Choosing $w=(x^*,\nabla \tf(\xu_N),0)$ and noting that $Ax^*=b$, $\tf(\xu_N) = \langle \xu_N, \nabla \tf(\xu_N)\rangle - \tf^*(\nabla \tf(\xu_N))$, and $\tf(x^*)\ge \langle x^*, \yu_N\rangle - \tf^*(\yu_N) $, by the definition of gap function $Q$ in \eqref{eq:Qhwk_SPP} we have
	\begin{align}
		Q(\wu_N, w) = & \mu\nu(\xu_N) + \langle \xu_N, \nabla \tf(\xu_N)\rangle - \tf^*(\nabla \tf(\xu_N)) - \mu\nu(x^*) - \langle x^*, \yu_N\rangle + \tf^*(\yu_N) 
		\\
		\ge & \mu\nu(\xu_N) + \tf(\xu_N) - \mu\nu(x^*) - \tf(x^*) = f(\xu_N) - f(x^*).
	\end{align}
	Combing the above result with \eqref{eq:Qest_proof} we obtain \eqref{eq:fest_proof} immediately. To prove \eqref{eq:Axbest_proof} let us study the gap functions involving $\wu_N$, $w^*$, and $\wu^*_N$. Observe that
	\begin{align}
		Q(\wu_N, w^*) = & \left[\mu\nu(\xu_N) + \langle \xu_N, y^*+\cA^\top  z^*\rangle - \tf^*(y^*) - \langle b, z^*\rangle\right] 
		\\
		& - \left[\mu\nu(x^*) + \langle x^*, \yu_N+\cA^\top \zu_N\rangle - \tf^*(\yu_N) - \langle b, \zu_N\rangle\right]
		\\
		\le & \mu\nu(\xu_N) + \langle \xu_N, y^*\rangle - \tf^*(y^*) - \mu\nu(x^*) - \langle x^*, \yu_N\rangle + \tf^*(\yu_N) 
		\\
		&+ \| \cA\xu_N -b\|_2\| z^*\|_2.
	\end{align}
	Here also observe that $Q(\wu_N, w^*) \ge 0$ since $w^*$ is a saddle point. By these observations and the definition of $\zu_N^*$ in \eqref{eq:wustar} we have
	\begin{align}
		 Q(\wu_N, \wu_N^*) = & \left[\mu\nu(\xu_N) + \langle \xu_N, y^*+\cA^\top \zu_N^*\rangle - \tf^*(y^*) - \langle b, \zu_N^*\rangle\right] 
		\\
		& - \left[\mu\nu(x^*) + \langle x^*, \yu_N+\cA^\top \zu_N\rangle - \tf^*(\yu_N) - \langle b, \zu_N\rangle\right]
		\\
		= & \mu\nu(\xu_N) + \langle \xu_N, y^*\rangle - \tf^*(y^*) - \mu\nu(x^*) - \langle x^*, \yu_N\rangle + \tf^*(\yu_N) 
		\\
		& + (\|z^*\|_2+1)\|\cA\xu_N -b\|_2
		\\
		\ge & Q(\wu_N, w^*) + \|\cA\xu_N - b\|_2\ge \|\cA\xu_N - b\|_2.
	\end{align}
	The above result, together with the bound of $Q$ in \eqref{eq:Qest_proof}, and the facts that $y^*=\nabla \tf(x^*)$, $z_0=0$, and 
	$
		\|z_0 - \zu_N^*\|_2^2 = \|\zu_N^*\|_2^2 = (\|z^*\|_2+1)^2,
	$
	yield the estimate \eqref{eq:Axbest_proof}.
\end{proof}

\vgap

Note that Proposition \ref{pro:Qest} is in fact a special case of Proposition \ref{pro:Qest_proof}. Indeed, since Algorithm \ref{alg:PDS_agent} is the agent view of Algorithm \ref{alg:PDS}, which is a special case of Algorithm \ref{alg:PDS_SPP} (when $h(z)\equiv 0$, $\cZ =\R^{md}$, $z_0=0$, $\cX$, $f$, and $\cA$ are defined in \eqref{eq:problem_decen}, and all norms are Euclidean). Therefore, the results in Proposition~\ref{pro:Qest} are immediately implied from Proposition \ref{pro:Qest_proof}.
%

Applying the results in Lemmas \ref{lem:Aest_S}, \ref{lem:Best}, and \ref{lem:CDest} to Proposition \ref{pro:Qest_raw} we can also prove Proposition \ref{pro:Qest_S}. The proof is similar to that of Proposition \ref{pro:Qest_proof} above.

\vgap

\begin{proof}[Proof of Proposition \ref{pro:Qest_S}]
	Denoting $\wu_N:=\left(\sum_{k=1}^{N}\beta_k\right)^{-1}\left(\sum_{k=1}^N\beta_k\hw_k\right)$, let us study the gap function $Q$ defined in \eqref{eq:Qhwk_SPP} in which $h(\cdot)\equiv 0$. Noting the convexity of $Q$, applying Proposition \ref{pro:Qest_raw} (in which $U(\cdot,\cdot) = \|\cdot-\cdot\|_2^2/2$, $V(\cdot,\cdot)=\|\cdot-\cdot\|_2^2/2$, and all norms are Euclidean) and the bounds of $A$ through $D$ in Lemmas \ref{lem:Aest_S} (with $\sigma_k=\sigma/\sqrt{c_k}$ as specified in Algorithm \ref{alg:SPDS}), \ref{lem:Best}, and \ref{lem:CDest} (in which $z_0=0$ from the description of Algorithm \ref{alg:SPDS_agent}), we have
	\begin{align}
		& \left(\tsum_{k=1}^{N}\beta_k\right)\E[Q(\wu_N, w)] \le \tsum_{k=1}^{N}\beta_k\E[Q(\hat w_k, w)]
		\\
		\le & 
		\E\left[
		\beta_{N} \langle x_{N-1} - \hx_N, y_N - y\rangle - \tfrac{\beta_Np_N}{4}\|x_{N-1} - \hx_N\|_2^2 + \tsum_{k=1}^{N}\tfrac{\beta_k \sigma^2}{p_kc_k}
		\right.
		\\
		& \qquad- \beta_N(\tau_N+1)W(y_N,y) 
		\\
		& -\tfrac{\beta_N}{T_N}\|\cA\| \|z_{N} - z\|_2\| x_N - x_N^{T_N-1}\|_2 - \tfrac{\beta_N\eta_N^{T_N}}{2T_N}\|x_N^{T_N-1} - x_N^{T_N}\|_2^2
		\\
		& + \tfrac{\beta_1q_1^1}{2T_1}\|z\|_2^2 - \tfrac{\beta_Nq_N^{T_N}}{2T_N}\|z_N - z\|_2^2
		\\
		& 
		\left.
		+ \tfrac{\beta_1}{2T_1}(\eta_1^1 + p_1T_1)\|x_0 - x\|_2^2 - \tfrac{\beta_N}{2T_N}(\mu + \eta_N^{T_N}+p_N)\|x_N - x\|_2^2
		\right].
	\end{align}	
	Here noting from \eqref{eq:cond_Qest_N}, \eqref{eq:cond_modified_S}, and \eqref{eq:Wbound} we can simplify the above to
	\begin{align}
		\left(\tsum_{k=1}^{N}\beta_k\right)\E[Q(\wu_N, w)]\le & \tfrac{\beta_1q_1^1}{2T_1}\|z\|_2^2 + \tfrac{\beta_1}{2T_1}(\eta_1^1 + p_1T_1)\|x_0 - x\|_2^2 + \tsum_{k=1}^{N}\tfrac{\beta_k \sigma^2}{p_kc_k}.
	\end{align}
	The remainder of the proof is similar to that of Proposition \ref{pro:Qest_proof} and hence is skipped.
\end{proof}

\section{Numerical experiments}
\label{sec:numerical}

In this section, we demonstrate the advantages of our proposed PDS method through some preliminary numerical experiments and compare it with the state-of-the-art communication-efficient decentralized method, namely the decentralized communication sliding (DCS) method proposed in \cite{lan2020communication}. 
We consider a decentralized convex smooth optimization problem of unregularized logistic regression model 
over a dataset that is not linearly separable.
In order to guarantee a fair comparison, all the implementation details described below are the same as suggested in \cite{lan2020communication}. In the linear constrained problem formulation \eqref{eq:problem_decen} of the decentralized problem we set $\cA=\cL\times I_d$ from the graph Laplacian. 
For the underlying communication network, we use the Erhos-Renyi algorithm
to generate three connected graphs with $m=100$ nodes as shown in Figure~\ref{network}. Note that nodes with different degrees are drawn in different colors, in particular, $G1$ has a maximum degree of $d_{max}=4$, $G2$ has $d_{max}=9$ and $G3$ has $d_{max}=20$. 
We also use the same dataset as in \cite{lan2020communication}, a real dataset "ijcnn1" from LIBSVM\footnote{This real dataset can be downloaded from \url{https://www.csie.ntu.edu.tw/~cjlin/libsvmtools/datasets/}.} and choose $20,000$ samples from this dataset as our problem instance data to train the decentralized logistic regression model. Since we have $m=100$ nodes (or agents) in the decentralized network, we evenly split these $20,000$ samples over $100$ nodes, and hence each network node has $200$ samples.

\begin{figure}[h]
\begin{minipage}{0.3\textwidth}
\centering
\includegraphics[scale = 0.11]{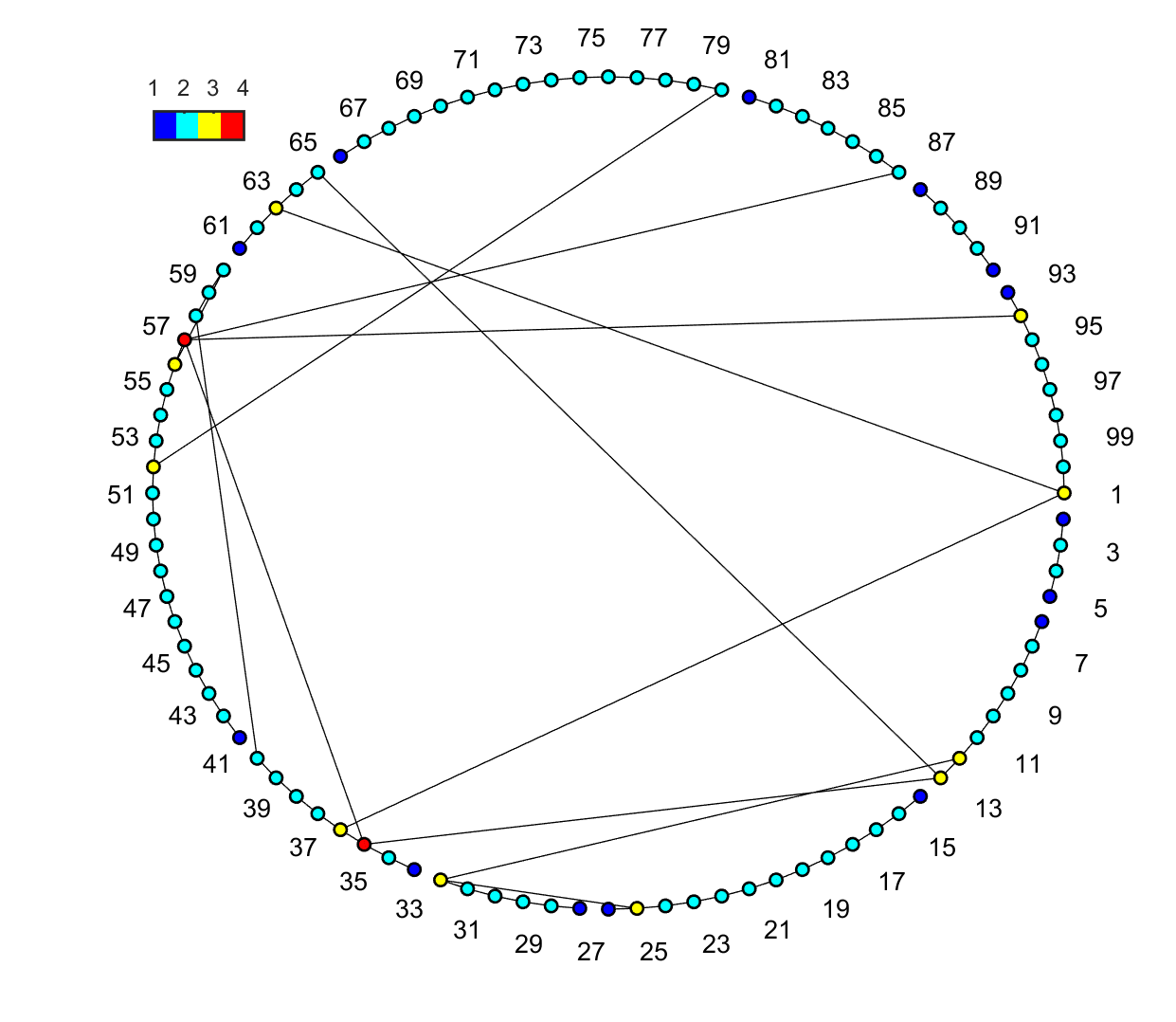}
\end{minipage}
\begin{minipage}{0.32\textwidth}
	\centering
	\includegraphics[scale = 0.15]{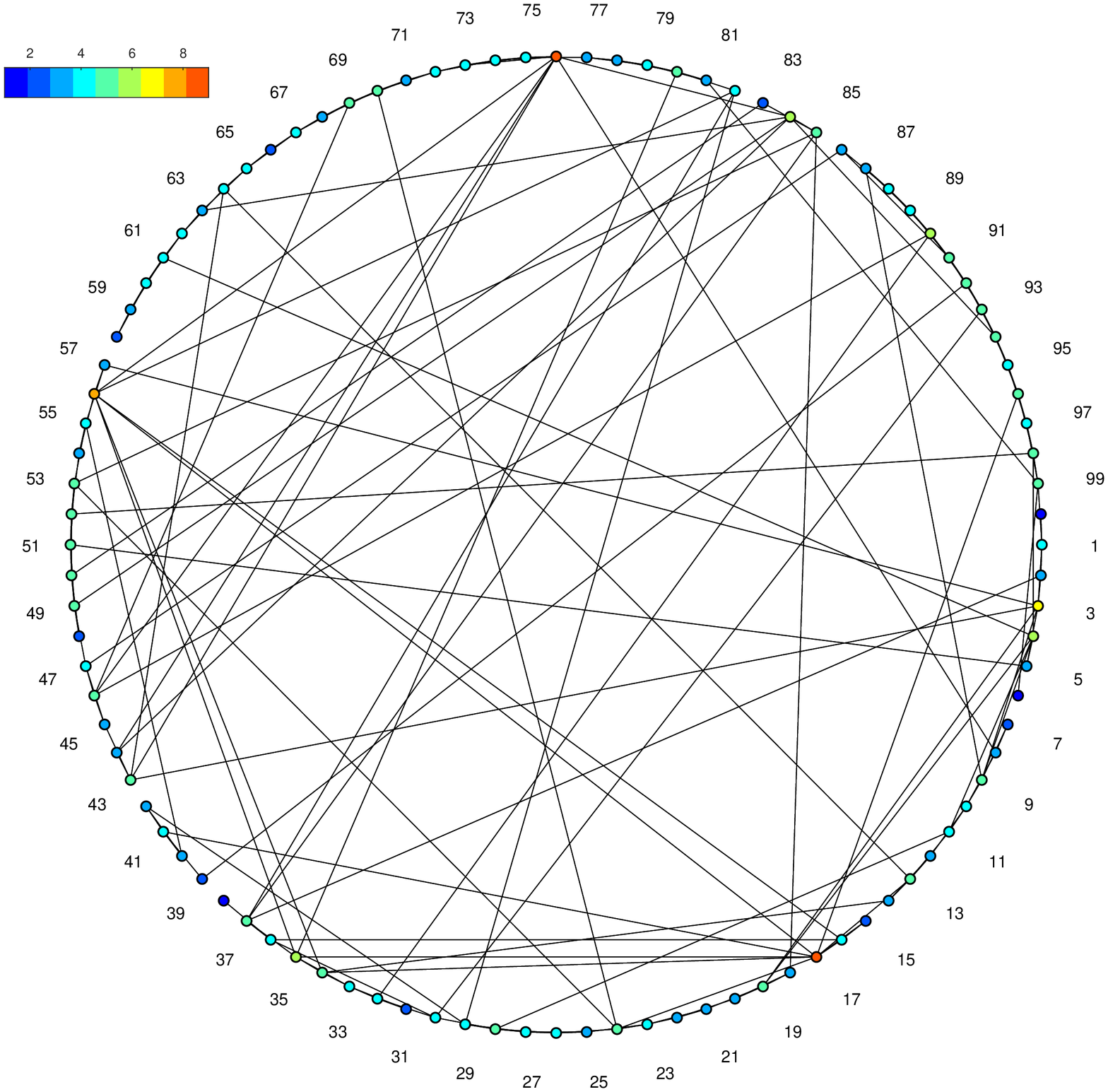}
\end{minipage}
\begin{minipage}{0.33\textwidth}
	\centering
	\includegraphics[scale = 0.12]{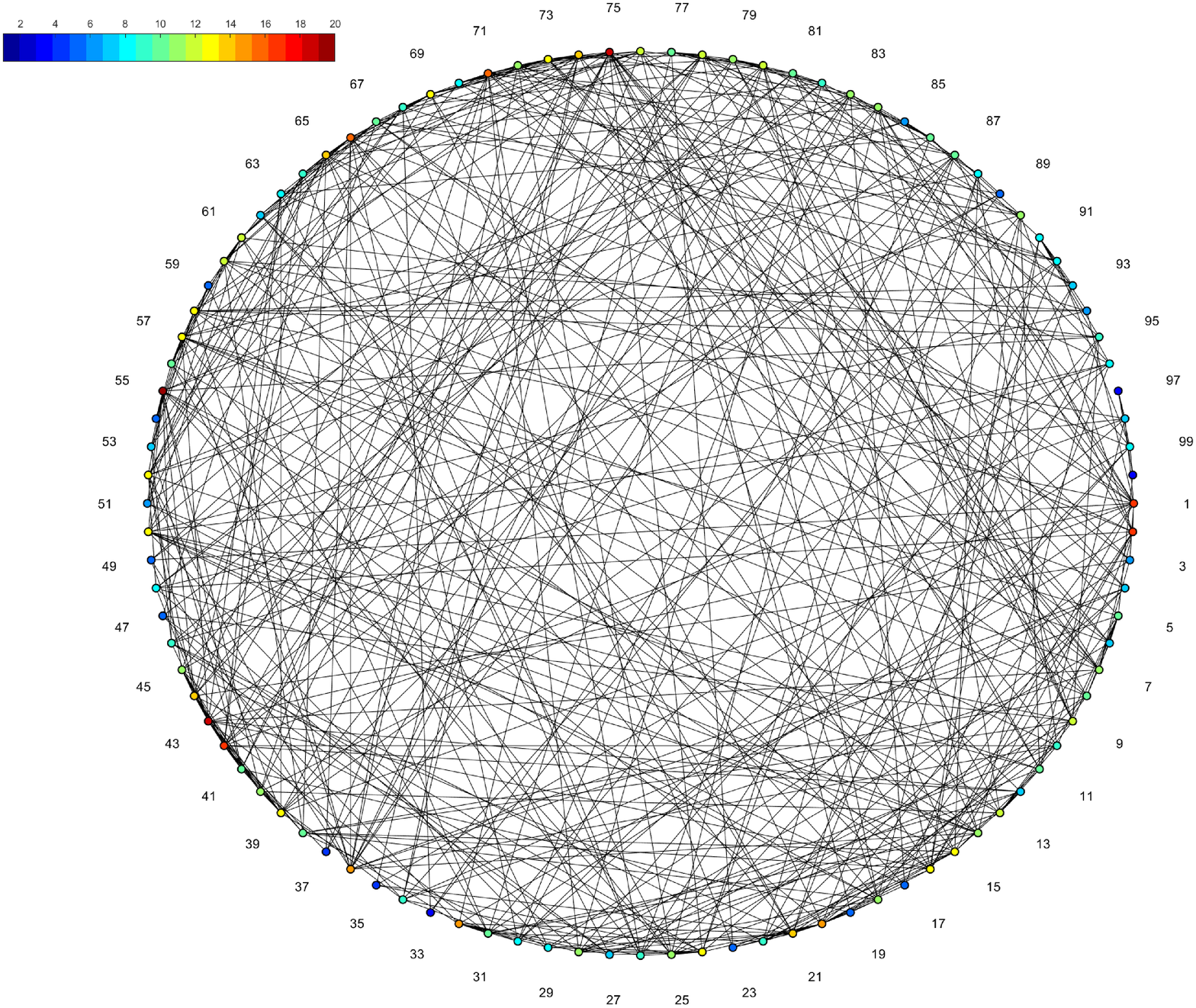}
\end{minipage}
\label{network}
\caption{\small The underlying decentralized networks $G1$, $G2$ and $G3$, from left to right. Nodes with different degrees are drawn in different colors.}
\end{figure}

All agents in the network start with the same initial point $x_0=0$ and $y_0=z_0=0$. We then compare the performances of the proposed PDS/SPDS algorithms with DCS (SDCS for stochastic case) proposed in \cite{lan2020communication} for solving \eqref{eq:problem_decen} with $f_i$ being the logistic loss 
function. We show the required number of communication rounds and (stochastic) gradient evaluations for the algorithms in comparison to obtain the same target loss.
In all problem instances, we use $\|\cdot\|_2$ norm in both primal and dual spaces. 
All the experiments are programmed in Matlab 2020a and run on Clemson University's Palmetto high-performance computing clusters (with $6$ Intel Xeon Gold CLX 6248R CPUs for a total of 168 cores).

In the deterministic case, for the DCS algorithm we use the parameter setting as suggested in  \cite{lan2020communication} including using a dynamic inner iteration limit as $\min\{10k, 5,000\}$ for possible performance improvement. For the PDS algorithm we use the parameter setting as suggested in Theorem~\ref{thm:spp} for solving smooth and convex problems ($\mu=0$). In particular, when setting the inner iteration limit $T_k$ we choose $R=1/(2\sqrt{2})$ in \eqref{eq:par_beforeDelta}
.
Note that we also try tuning the estimation of Lipschitz constant $\tilde L$ for the best performance of the PDS algorithm.

\begin{table}[H]
\caption{Comparison of the DCS and PDS algorithms in terms of reaching the same target loss}
\label{tab:res1}
\begin{center}
\small
\begin{tabular}{c|c|c|c|c|c}
    Algorithm & Graph & Target Loss & Achieved $\|\mathcal{A}x\|$ & \begin{tabular}[c]{@{}l@{}}Com. \\ rounds\end{tabular} & \begin{tabular}[c]{@{}c@{}}Gradient \\ evaluations\end{tabular} \\
    \hline
DCS & G1($d_{max}=4$) & $70$  & $4.94e-04$   & $3,110$                      & $6,527,500$        \\
\hline
PDS  &  G1($d_{max}=4$)  & $70$ & $8.95e-02$         & $154$   & $24$   \\ \hline
DCS & G2($d_{max}=9$)     & $70$   &  $3.22e-14$ &  $5,624$ & $12,812,500$  \\ \hline
PDS  & G2($d_{max}=9$)    & $70 $ & $2.08e-01$         & $274$   & $25$    \\ \hline
DCS  & G3($d_{max}=20$) &$70$  &   NA\footnotemark[3]             & NA     & NA                                  \\ \hline
PDS & G3($d_{max}=20$) & $70$  & $1.84e-02$         & $468$   & $24$               \\ \hline \hline
DCS & G1($d_{max}=4$) & $60$  &   NA \footnotemark[3]            & NA     & NA                                  \\
\hline
PDS &  G1($d_{max}=4$)  & $60$ &  $3.60e-01$        & $236$   &  $60$  \\ \hline
DCS & G2($d_{max}=9$)    & $60$   &   NA \footnotemark[3]            & NA     & NA                               \\ \hline
PDS & G2($d_{max}=9$)    & $60 $ & $9.15e-01$  & $340$ &   $58$              \\ \hline
DCS  & G3($d_{max}=20$) &$60$  &   NA \footnotemark[3]            & NA     & NA                                  \\ \hline
PDS & G3($d_{max}=20$) & $60$   & $2.09e-01$  & $564$ & $54$  
\end{tabular}
\end{center}
\end{table}
\footnotetext{We use ``NA'' for DCS experiments running more than $8,000$ communication rounds but not achieving the target losses.}

Table~\ref{tab:res1} show the results we obtained from the experiments of solving deterministic logistic regression problems. From the table, to reach the same target loss our proposing PDS algorithm requires less rounds of communication and much less gradient evaluations than the DCS algorithm in \cite{lan2020communication}. 
In particular, the number of gradient evaluations barely changes as the maximum degree of the graph increases, which matches our theoretic results on graph topology invariant gradient complexity bounds. 
Moreover, we can see that both algorithms achieves reasonable feasibility residue, which is measured by $\|\mathcal{A}x\|_2$.
It needs to be pointed out that the DCS algorithm achieves a smaller feasibility residue but a much higher loss, while the PDS algorithm spends more efforts in reducing the loss value and maintaining an acceptable feasibility residue. 
Therefore, we can conclude that the PDS algorithm maintains a better trade-off between loss and feasibility residue than the DCS algorithm.

We also consider the stochastic problem in which the algorithms can only access the unbiased stochastic samples of the gradients. 
For the SDCS algorithm, we use the parameter setting as suggested in \cite{lan2020communication} (see Theorem 4 within) except using a dynamic inner iteration limit as $\min\{10k, 5,000\}$ for possible performance improvement. For the proposed SPDS algorithm, we use the parameter setting as suggested in Theorem~\ref{thm:spp_S} for solving smooth and convex stochastic problems with $R=1$ and $c=1/4$ in \eqref{eq:par_beforeDelta_S}. 
We set $N$ to be the smallest iteration limit where $N \mod 10 =0$, and hence $N=30$ for the cases when we set target loss as $70$ and $N=100$ when target loss is $60$.
\begin{table}[H]
\caption{Comparison of the SDCS and SPDS algorithms in terms of reaching the same target loss}
\label{tab:res2}
\begin{center}
\small
\begin{tabular}{c|c|c|c|c|c}
    Algorithm & Graph & Target Loss & Achieved $\|\mathcal{A}x\|$ & \begin{tabular}[c]{@{}l@{}}Com. \\ rounds\end{tabular} & \begin{tabular}[c]{@{}c@{}}Stochastic \\ grads.\end{tabular} \\
    \hline
SDCS & G1($d_{max}=4$) & $70$  &  NA\footnotemark[4]   &   NA                    &    NA    \\
\hline
SPDS  &  G1($d_{max}=4$)  & $70$ & $6.10e-04$         & $428$   & $431$ \\ \hline
SDCS & G2($d_{max}=9$)     & $70$   &  NA\footnotemark[4]   &   NA                    &    NA   \\ \hline
SPDS & G2($d_{max}=9$) & $70$  & $3.77e-04$         & $740$   & $431$     \\ \hline

SDCS  & G3($d_{max}=20$) &$70$  &   NA\footnotemark[4]             & NA     & NA                                  \\ \hline
SPDS  & G3($d_{max}=20$) & $70$  & $6.12e-05$         & $1,418$   & $431$               \\ \hline\hline
SDCS & G1($d_{max}=4$) & $60$  &   NA \footnotemark[4]            & NA     & NA                                  \\
\hline
SPDS &  G1($d_{max}=4$)& $60$  & $9.68e-04$         & $506$   & $1,434$   \\ \hline
SDCS & G2($d_{max}=9$)    & $60$   &   NA \footnotemark[4]            & NA     & NA                               \\ \hline
SPDS  & G2($d_{max}=9$)    & $60 $ & $8.87e-04$  & $816$ &   $1,433$              \\ \hline
SDCS  & G3($d_{max}=20$) &$60$  &   NA \footnotemark[4]            & NA     & NA                                  \\ \hline
SPDS & G3($d_{max}=20$) & $60$   & $1.18e-04$   & $1,496$ & $1,433$  
\end{tabular}
\end{center}
\end{table}
\footnotetext{We use ``NA'' for SDCS experiments running more than $4,000$ communication rounds but not achieving the target losses.}

Table~\ref{tab:res2} show the results we obtained from the experiments of solving logistic regression problems using stochastic samples of gradients. Observe from the table that the number of required stochastic gradient samples for the SPDS algorithm do not change as the maximum degree of the graph increases, which matches our theoretic result that the sampling complexity of the SPDS algorithm is graph topology invariant for stochastic problems.
The SPDS algorithm can achieve target losses within reasonable amount of CPU time and algorithm iterations, while the SDCS algorithm can not achieve any of the target loss within the required communication rounds. 
Therefore, we can conclude that the proposed SPDS algorithm outperforms the SDCS algorithm in \cite{lan2020communication} in terms of both sampling and communication complexities.

\section{Concluding remarks}
\label{sec:conclusion}

In this paper, we present a new class of algorithms for solving a class of decentralized multi-agent optimization problem. Our proposed primal dual sliding (PDS) algorithm is able to compute an approximate solution to the general convex smooth deterministic problem with $\cO\left(\sqrt{\tL/\varepsilon}+\|\cA\|/\varepsilon\right)$ communication rounds, which matches the current state-of-the-art. However, the number of gradient evaluations required by the PDS algorithm is improved to $\cO\left(\sqrt{\tL/\varepsilon}\right)$ and is invariant with respect to the graph topology. To the best of our knowledge, this is the only decentralized algorithm whose gradient complexity is graph topology invariant when solving decentralized problems over a constraint feasible set. We also propose a stochastic primal dual sliding (SPDS) algorithm that is able to compute an approximate solution to the general convex smooth stochastic problem with $\cO\left(\sqrt{\tL/\varepsilon} + \sigma^2/\varepsilon^2\right)$ sampling complexity, which is also the only algorithm in the literature that has graph topology invariant sampling complexity for solving problems of form \eqref{eq:problem_multi_agent}. Similar convergence results of the PDS and SPDS algorithms are also developed for strongly convex smooth problems. 


\section*{Acknowedgement}
The authors would like to thank Huan Li for helpful discussions on complexity bounds of unconstrained decentralized optimization problems.

\bibliographystyle{siamplain}
\bibliography{yuyuan}

\begin{thebibliography}{10}

\bibitem{alghunaim2020decentralized}
{\sc S.~A. Alghunaim, E.~Ryu, K.~Yuan, and A.~H. Sayed}, {\em Decentralized
  proximal gradient algorithms with linear convergence rates}, IEEE
  Transactions on Automatic Control,  (2020).

\bibitem{aybat2017distributed}
{\sc N.~S. Aybat, Z.~Wang, T.~Lin, and S.~Ma}, {\em Distributed linearized
  alternating direction method of multipliers for composite convex consensus
  optimization}, IEEE Transactions on Automatic Control, 63 (2017), pp.~5--20.

\bibitem{boyd2011distributed}
{\sc S.~Boyd, N.~Parikh, E.~Chu, B.~Peleato, and J.~Eckstein}, {\em Distributed
  optimization and statistical learning via the alternating direction method of
  multipliers}, Foundations and Trends{\textregistered} in Machine Learning, 3
  (2011), pp.~1--122.

\bibitem{chang2014multi}
{\sc T.-H. Chang, M.~Hong, and X.~Wang}, {\em Multi-agent distributed
  optimization via inexact consensus admm}, IEEE Transactions on Signal
  Processing, 63 (2014), pp.~482--497.

\bibitem{chen2018lag}
{\sc T.~Chen, G.~Giannakis, T.~Sun, and W.~Yin}, {\em Lag: Lazily aggregated
  gradient for communication-efficient distributed learning}, in Advances in
  Neural Information Processing Systems, 2018, pp.~5050--5060.

\bibitem{duchi2011dual}
{\sc J.~C. Duchi, A.~Agarwal, and M.~J. Wainwright}, {\em Dual averaging for
  distributed optimization: Convergence analysis and network scaling}, IEEE
  Transactions on Automatic control, 57 (2011), pp.~592--606.

\bibitem{durham2012distributed}
{\sc J.~W. Durham, A.~Franchi, and F.~Bullo}, {\em Distributed pursuit-evasion
  without mapping or global localization via local frontiers}, Autonomous
  Robots, 32 (2012), pp.~81--95.

\bibitem{forero2010consensus}
{\sc P.~A. Forero, A.~Cano, and G.~B. Giannakis}, {\em Consensus-based
  distributed support vector machines.}, Journal of Machine Learning Research,
  11 (2010).

\bibitem{ghadimi2012optimal}
{\sc S.~Ghadimi and G.~Lan}, {\em Optimal stochastic approximation algorithms
  for strongly convex stochastic composite optimization i: A generic
  algorithmic framework}, SIAM Journal on Optimization, 22 (2012),
  pp.~1469--1492.

\bibitem{ghadimi2013optimal}
{\sc S.~Ghadimi and G.~Lan}, {\em Optimal stochastic approximation algorithms
  for strongly convex stochastic composite optimization, {II}: shrinking
  procedures and optimal algorithms}, SIAM Journal on Optimization, 23 (2013),
  pp.~2061--2089.

\bibitem{hong2017stochastic}
{\sc M.~Hong and T.-H. Chang}, {\em Stochastic proximal gradient consensus over
  random networks}, IEEE Transactions on Signal Processing, 65 (2017),
  pp.~2933--2948.

\bibitem{jadbabaie2003coordination}
{\sc A.~Jadbabaie, J.~Lin, and A.~S. Morse}, {\em Coordination of groups of
  mobile autonomous agents using nearest neighbor rules}, IEEE Transactions on
  automatic control, 48 (2003), pp.~988--1001.

\bibitem{koloskova2019decentralized}
{\sc A.~Koloskova, S.~Stich, and M.~Jaggi}, {\em Decentralized stochastic
  optimization and gossip algorithms with compressed communication}, in
  International Conference on Machine Learning, 2019, pp.~3478--3487.

\bibitem{kovalev2020optimal}
{\sc D.~Kovalev, A.~Salim, and P.~Richt{\'a}rik}, {\em Optimal and practical
  algorithms for smooth and strongly convex decentralized optimization},
  Advances in Neural Information Processing Systems, 33 (2020).

\bibitem{lan2016gradient}
{\sc G.~Lan}, {\em Gradient sliding for composite optimization}, Mathematical
  Programming, 159 (2016), pp.~201--235.

\bibitem{lan2020communication}
{\sc G.~Lan, S.~Lee, and Y.~Zhou}, {\em Communication-efficient algorithms for
  decentralized and stochastic optimization}, Mathematical Programming, 180
  (2020), pp.~237--284.

\bibitem{lan2016accelerated}
{\sc G.~Lan and Y.~Ouyang}, {\em Accelerated gradient sliding for structured
  convex optimization}, arXiv preprint arXiv:1609.04905,  (2016).

\bibitem{li2020decentralized}
{\sc H.~Li, C.~Fang, W.~Yin, and Z.~Lin}, {\em Decentralized accelerated
  gradient methods with increasing penalty parameters}, IEEE Transactions on
  Signal Processing, 68 (2020), pp.~4855--4870.

\bibitem{li2020optimal}
{\sc H.~Li, Z.~Lin, and Y.~Fang}, {\em Optimal accelerated variance reduced
  extra and diging for strongly convex and smooth decentralized optimization},
  arXiv preprint arXiv:2009.04373,  (2020).

\bibitem{liu2019communication}
{\sc Y.~Liu, W.~Xu, G.~Wu, Z.~Tian, and Q.~Ling}, {\em Communication-censored
  admm for decentralized consensus optimization}, IEEE Transactions on Signal
  Processing, 67 (2019), pp.~2565--2579.

\bibitem{mokhtari2016decentralized}
{\sc A.~Mokhtari, W.~Shi, Q.~Ling, and A.~Ribeiro}, {\em A decentralized
  second-order method with exact linear convergence rate for consensus
  optimization}, IEEE Transactions on Signal and Information Processing over
  Networks, 2 (2016), pp.~507--522.

\bibitem{mokhtari2016dqm}
{\sc A.~Mokhtari, W.~Shi, Q.~Ling, and A.~Ribeiro}, {\em Dqm: Decentralized
  quadratically approximated alternating direction method of multipliers}, IEEE
  Transactions on Signal Processing, 64 (2016), pp.~5158--5173.

\bibitem{nedic2014distributed}
{\sc A.~Nedi{\'c} and A.~Olshevsky}, {\em Distributed optimization over
  time-varying directed graphs}, IEEE Transactions on Automatic Control, 60
  (2014), pp.~601--615.

\bibitem{nedic2018network}
{\sc A.~Nedi{\'c}, A.~Olshevsky, and M.~G. Rabbat}, {\em Network topology and
  communication-computation tradeoffs in decentralized optimization},
  Proceedings of the IEEE, 106 (2018), pp.~953--976.

\bibitem{nedic2017fast}
{\sc A.~Nedi{\'c}, A.~Olshevsky, and C.~A. Uribe}, {\em Fast convergence rates
  for distributed non-bayesian learning}, IEEE Transactions on Automatic
  Control, 62 (2017), pp.~5538--5553.

\bibitem{nedic2009distributed}
{\sc A.~Nedic and A.~Ozdaglar}, {\em Distributed subgradient methods for
  multi-agent optimization}, IEEE Transactions on Automatic Control, 54 (2009),
  pp.~48--61.

\bibitem{nesterov2005smooth}
{\sc Y.~Nesterov}, {\em Smooth minimization of non-smooth functions},
  Mathematical programming, 103 (2005), pp.~127--152.

\bibitem{nesterov2004introductory}
{\sc Y.~E. Nesterov}, {\em Introductory Lectures on Convex Optimization: A
  Basic Course}, Kluwer Academic Publishers, Massachusetts, 2004.

\bibitem{rabbat2004distributed}
{\sc M.~Rabbat and R.~Nowak}, {\em Distributed optimization in sensor
  networks}, in Proceedings of the 3rd international symposium on Information
  processing in sensor networks, 2004, pp.~20--27.

\bibitem{ram2009distributed}
{\sc S.~S. Ram, V.~Veeravalli, and A.~Nedi{\'c}}, {\em Distributed
  non-autonomous power control through distributed convex optimization}, in
  28th Conference on Computer Communications, IEEE INFOCOM 2009, 2009,
  pp.~3001--3005.

\bibitem{shi2014linear}
{\sc W.~Shi, Q.~Ling, K.~Yuan, G.~Wu, and W.~Yin}, {\em On the linear
  convergence of the admm in decentralized consensus optimization}, IEEE
  Transactions on Signal Processing, 62 (2014), pp.~1750--1761.

\bibitem{terelius2011decentralized}
{\sc H.~Terelius, U.~Topcu, and R.~M. Murray}, {\em Decentralized multi-agent
  optimization via dual decomposition}, IFAC proceedings volumes, 44 (2011),
  pp.~11245--11251.

\bibitem{tsianos2012consensus}
{\sc K.~I. Tsianos, S.~Lawlor, and M.~G. Rabbat}, {\em Consensus-based
  distributed optimization: Practical issues and applications in large-scale
  machine learning}, in 2012 50th annual allerton conference on communication,
  control, and computing (allerton), IEEE, 2012, pp.~1543--1550.

\bibitem{tsitsiklis1986distributed}
{\sc J.~Tsitsiklis, D.~Bertsekas, and M.~Athans}, {\em Distributed asynchronous
  deterministic and stochastic gradient optimization algorithms}, IEEE
  transactions on automatic control, 31 (1986), pp.~803--812.

\bibitem{wei20131}
{\sc E.~Wei and A.~Ozdaglar}, {\em On the {O}(1/k) convergence of asynchronous
  distributed alternating direction method of multipliers}, in 2013 IEEE Global
  Conference on Signal and Information Processing, IEEE, 2013, pp.~551--554.

\bibitem{ye2020pmgt}
{\sc H.~Ye, W.~Xiong, and T.~Zhang}, {\em Pmgt-vr: A decentralized
  proximal-gradient algorithmic framework with variance reduction}, arXiv
  preprint arXiv:2012.15010,  (2020).

\end{thebibliography}
\end{document}